\documentclass[8pt]{article}
\usepackage[utf8]{inputenc}
\usepackage[usenames,dvipsnames,svgnames,table]{xcolor}	
\usepackage[]{amsmath}
\usepackage{cases}
\usepackage{empheq}
\usepackage{cleveref} 
\usepackage{amssymb}
\usepackage{bbm}
\usepackage{enumerate}   
\usepackage{geometry}
\usepackage{amsfonts}
\usepackage{color}
\usepackage[]{amsthm} 
 \usepackage[shortlabels]{enumitem}
 \usepackage{cleveref}

\newtheorem{theorem}{Theorem}[section]
\newtheorem{lemma}[theorem]{Lemma}
\newtheorem{cor}[theorem]{Corollary}
\newtheorem{proposition}[theorem]{Proposition}
\newtheorem{mydef}[theorem]{Definition}

\newtheorem{remark}[theorem]{Remark}
\newtheorem*{hyp}{\textit{Hypothesis}}
\def\E{\mathbb{E}}
\def\R{\mathbb{R}}
\def\P{\mathbb{P}}
\def\Q{\mathbb{Q}}

\def\B{\mathcal{B}}
\def\1{\mathbbm{1}}
\def\F{\mathcal{F}}

\usepackage{xspace}

\title{A new McKean-Vlasov stochastic interpretation of the 
parabolic-parabolic Keller-Segel model: The one-dimensional case. }
\author{Denis Talay\footnote{TOSCA, Inria Sophia 
Antipolis-Méditerranée, 2004 Route des Lucioles, 06902 Valbonne, France 
(denis.talay@inria.fr, milica.tomasevic@inria.fr).} and Milica 
Toma\v{s}evi\'c \footnotemark[\value{footnote}] \footnote{Universit\'e 
C\^ote d'Azur, Campus Valrose, Batiment M, 28 Avenue de Valrose, 06108 
Nice CEDEX 2, France. }}
\date{}

\begin{document}
\maketitle
\noindent
\textbf{Abstract:} In this paper we analyze a stochastic interpretation of the one-dimensional parabolic-parabolic Keller-Segel system without cut-off.
It involves an original type of McKean-Vlasov interaction kernel. At the particle level,
each particle interacts with all the past of each other particle by means of a time integrated functional involving a singular kernel. At the mean-field level
studied here, the McKean-Vlasov limit process interacts
with all the past time marginals of its probability distribution in a similarly singular way. We prove that the 
parabolic-parabolic Keller-Segel system in the whole Euclidean space and the corresponding McKean-Vlasov stochastic differential equation
are well-posed for any values of the parameters of the model.
\textbf{Classification:} 60H30 60H10 60K35.
%\tableofcontents
%\pagebreak

%%%%%%%%%%%%%%%%%%%%%%%%%%%%%%%%%%%%%%%%%%%%%%%%%%%%%%%%%%%%

\section{Introduction}
%\section{Application to the one-dimensional Keller-Segel model}

The standard $d$-dimensional parabolic--parabolic Keller--Segel model 
for chemotaxis describes  the time evolution of the density $\rho_t$ of 
a cell population and of the concentration $c_t$ of a chemical 
attractant: 
\begin{equation}
\label{KSd}
\begin{cases}
&  \partial_t \rho(t,x)=\nabla \cdot (\frac{1}{2}\nabla \rho -  
\chi\rho \nabla 
c)(t,x), \quad t>0, \ x \in \R^d, \\
 & \alpha~\partial_t c(t,x) = \frac{1}{2}\triangle c(t,x)  -\lambda 
 c(t,x) + \rho(t,x),  \quad 
 t>0, \ x \in \R^d.\\
 & \rho(0,x)=\rho_0(x),~~~c(0,x)=c_0(x), 
\end{cases}
\end{equation}  
See e.g. Corrias~\cite{Corrias2014}, Perthame~\cite{perthame} and 
references therein for theoretical results on this system of PDEs and 
applications to Biology. 

Recently, stochastic interpretations have been proposed for a 
simplified version of the model, that is, 
the parabolic-elliptic model which corresponds to the value
$\alpha=0$. They all rely on the fact that, in the parabolic-elliptic 
case, the equations for $\rho_t$ and $c_t$ can be decoupled and $c_t$ 
can be explicited as the convolution of the initial condition~$c_0$ and 
the kernel $k(x)=-\frac{x}{2 \pi |x|^2}$. Consequently, the stochastic process of 
McKean--Vlasov type whose $\rho_t$ is the time marginal density 
involves the singular interaction kernel~$k$. This explains why, so 
far, only partial results are obtained and heavy techniques are used to 
get them. In ~Jabir et al.~\cite{JTT}, one may find a short review of the works by  Ha\v{s}kovec and Schmeiser~\cite{haskovec-schmeiser}, Fournier and 
Jourdain~\cite{fournier-jourdain} and Cattiaux and P\'ed\`eches~\cite{cattiaux-pedeches}.

Budhiraja and Fan \cite{Budhiraja} have studied a McKean--Vlasov SDE related to a parabolic--parabolic version of the model with cut-off
%of the model with a smooth coupling between $\rho$ and $c$,
  %study a particle system with a smooth time integrated kernel and
%prove it propagates chaos. Moreover, adding a
 and a forcing potential term. Under a suitable convexity assumption, they obtain 
uniform in time concentration 
inequalities for the corresponding particle system and uniform in time error estimates for a numerical approximation of the exact McKean--Vlasov process.
% introduce a 
%kernel with a cut-off parameter and obtain the tightness of the 
%particle probability distributions w.r.t. the cut-off parameter and the 
%number of particles. They also obtain partial results in the direction 
%of the propagation of chaos. In the 
%subcritical case, that is, when the parameter~$\chi$ of the 
%parabolic-elliptic model is small enough, Fournier and 
%Jourdain~\cite{fournier-jourdain} obtain that a
%particle system without cut-off is well--posed. In addition, they 
%obtain a consistency property which is weaker than the propagation of 
%chaos. They also describe complex 
%behaviours of the particle system in the sub and super critical cases.
%Cattiaux and P\'ed\`eches~\cite{cattiaux-pedeches} revisit the 
%well-posedness of the particle system by using Dirichlet forms.

We here deal with the parabolic--parabolic system ($\alpha >0$) without cut-off and study the McKean-Vlasov stochastic representation of the mild formulation of the equation satisfied by~$\rho_t$.
%This representation has the advantage 
%to concern the parabolic-parabolic model ($\alpha >0$).
 This representation involves a singular 
interaction kernel which is 
%quite
 different from the one in the above mentioned approaches and does not seem to have been studied in the McKean-Vlasov non-linear SDE literature. The 
system reads
\begin{equation}
\label{NLSDEd}
\begin{cases}
& dX_t= b^\sharp(t,X_t)dt +\Big\{ \int_0^t (K^\sharp_{t-s}\ast 
p_s)(X_t)ds\Big\} dt
+  dW_t, \quad t> 0, \\%Old 15 12: \quad t\leq  T, \\
& p_s(y)dy:=  \mathcal{L}( X_{s}),\quad X_0 \sim \rho_0(x)dx,
\end{cases}
\end{equation}
where $K^\sharp_t(x):=\chi e^{-\lambda t}\nabla( \frac{1}{(2\pi 
t)^{d/2}}e^{-\frac{|x|^2}{2t}})$ and $b^\sharp(t,x):=\chi e^{-\lambda t}  
\nabla \E c_0(x+W_t)$.
Here, $(W_t)_{t\geq 0}$ is a $d$-dimensional Brownian motion on a 
filtered probability space $(\Omega, \F, \P, (\F_t) )$ and  $X_0$ is an 
$\R^d$-valued $\F_0-$measurable random variable. 
Notice that the formulation requires that the one dimensional time 
marginals of the law of the solution are absolutely continuous with 
respect to Lebesgue's measure and that the process interacts 
with all the past time marginals of its probability distribution through a functional involving a singular kernel.

The analysis of the well-posedness of this non-linear equation and the 
proof that $p_s=\rho(s,\cdot)$ for any~$s$ are delicate, particularly in 
the multi-dimensional case when $\chi$ is large 
enough to induce solutions with blow-ups in finite time. This 
theoretical work is 
still in progress~\cite{Mi-De-2D}. As numerical simulations of 
the related particle system appear to be effective,
it seems interesting to validate our approach in the 
one-dimensional case.

The objective of this paper is to prove general 
existence and 
uniqueness results for both the deterministic system~(\ref{KSd}) and the 
stochastic dynamics~(\ref{NLSDEd}) in $d=1$. In our companion 
paper~\cite{JTT} we show the well-posedness and 
propagation of chaos property of the corresponding particle system  where each particle interacts with all the past of the other ones by 
means of a time integrated singular kernel.  

In this one-dimensional framework the PDE \eqref{KSd} was previously studied by~\cite{OsakiYagi,HillenPotapov} in bounded intervals~$I$ with 
periodic boundary conditions while we here deal with the problem posed
on the whole space~$\R$. In~\cite{OsakiYagi} one assumes $\rho_0 
\in L^2(I)\cap L^1(I)$ and $c_0 \in H^1(I)$. In~\cite{HillenPotapov} 
one assumes $\rho_0 \in 
L^\infty(I)\cap L^1(I)$ and $c_0 \in W^{\sigma,p}(I)$, where $p$ and 
$\sigma$ 
belong to a particular set of parameters. Here, we only suppose 
that $\rho_0$ is in $L^1(\R)$.

We emphasize that we do not limit ourselves to the specific
kernel~$K^\sharp_t(x)$ related to the Keller--Segel model. We below 
show that the mean--field PDE and stochastic differential equation of 
Keller-Segel type are well-posed for a whole class of time integrated 
singular kernels. The mean-field SDE cannot be analyzed by means of 
standard coupling methods or Wasserstein distance contractions. Both to 
construct local solutions and to go from local to global solutions, an 
important issue consists in properly defining the set of weak solutions 
without any assumption on the initial density $\rho_0$, which led us to 
introduce constraints on the time marginal densities.  To prove that 
these constraints are satisfied in the limit of an iterative procedure 
(where the kernel is not cut off), 
%Girsanov theorem, 
%Krylov-R\"{o}ckner's method or PDE analysis results cannot be applied because
 the norms of the successive time marginal densities  cannot be 
allowed to exponentially depend on the $L^\infty$-norm of the 
successive corresponding drifts. They neither can be allowed to depend on H\"{o}lder-norms of the drifts. Therefore, we use an accurate estimate (with explicit constants)
on densities of one-dimensional diffusions with bounded measurable 
drifts which is obtained by a stochastic technique rather than the PDE techniques. This strategy 
allows us to get uniform bounds on the
sequence of drifts, which is essential to get existence and uniqueness of the local solution to the non-linear martingale problem solved by any limit of the Picard procedure, and to suitably paste local solutions when constructing the global solution.

The paper is organized as follows. In Section~2 we state our main 
results. In Section~3 we prove a preliminary estimate on the 
probability density of diffusions whose drift is only supposed Borel 
measurable and bounded. In Section~4 we study a non-linear 
McKean-Vlasov-Fokker-Planck equation. In Section~5 we prove the local 
existence and uniqueness of a solution to a non-linear stochastic
differential equation more general than~\eqref{NLSDEd} (for $d=1$). In 
Section~6 we get the global well-posedness of this equation. In 
Section~7 we apply the preceding result to the specific case of the 
one-dimensional parabolic--parabolic Keller-Segel model. The appendix 
section~8 concerns an 
explicit formula for the transition density of a particular diffusion.

\paragraph*{Notation.} In all the paper we denote by~$C_T,~ C_T(b_0, p_0)$, etc.,  any constant 
which depends on $T$ and the other specified parameters, but is uniform w.r.t. $t\in [0,T]$ and may change from line.
\section{Our main results}

Our first main result concerns the well-posedness of a non-linear 
one-dimensional stochastic differential equation (SDE) with a non 
standard 
McKean--Vlasov interaction kernel which at each time~$t$ involves in a 
singular way all the time marginals up to time~$t$ of the probability 
distribution of the solution. As our technique of analysis is not 
limited to the above kernel~$K^\sharp$, we consider the following 
McKean-Vlasov stochastic equation: 
\begin{equation}
\label{NLSDEsimple}
\begin{cases}
& dX_t= b(t,X_t)dt +\Big\{\int_0^t (K_{t-s}\ast p_s)(X_t)ds\Big\} dt 
+ dW_t, \quad t\leq  T, \\
& p_s(y)dy:=  \mathcal{L}( X_{s}),\quad X_0 \sim p_0,
\end{cases}
\end{equation}
and in all the sequel we assume the following conditions on the 
interaction kernel.

\begin{hyp}[H]
 The function $K$ defined on $\R^+ \times \R$ is such that for any 
 $T>0$:
\begin{enumerate}
\item For any $t>0$, $K_t$ is in $L^1(\R)$.
%\item $K\in L^1((0,T];L^1(\R))\cap L^1((0,T]; L^2(\R))$.
\item For any $t>0$ the function $K_t(x)$ is a bounded continuous 
function on 
$\R$.

\item The set of points $x\in\R$ such that $\lim_{t\rightarrow0}K_t(x) < \infty $ has full Lebesgue measure. 

\item For any $t>0$,  the function $f_1(t):=\int_0^t 
\frac{\|K_{t-s}\|_{L^1(\R)}}{\sqrt{s}} ds$ is well defined and 
bounded on~$[0,T]$.
%\item For any $t>0$,  the functions $f_1(t):=\int_0^t 
%\frac{\|K_{t-s}\|_{L^1(\R)}}{\sqrt{s}} ds$ and $f_2(t):=\int_0^t 
%\frac{\|K_{t-s}\|_{L^2(\R)}}{s^{1/4}} ds$ are well defined and 
%bounded on~$[0,T]$.
\item For any $T>0$ there exists $C_T$ such that, for any 
probability density $\phi$ on $\R$,
$$ \sup_{(t,x)\in (0,T]\times 
\R}\int \phi(y) \|K_\cdot(x-y)\|_{L^1(0,t)}~dy\leq 
C_T.$$
\item Finally,
$$ \sup_{0\leq t\leq T} \int_0^T \| K_{T+t-s} 
\|_{L^1(\R)}~\frac{1}{\sqrt{s}}~ds \leq C_T. $$
  % fin de rouge
\end{enumerate}
\end{hyp}

As emphasized in the introduction, the well-posedness of 
the system~\eqref{NLSDEsimple}
cannot be obtained by applying known results in the literature.
%\milica{ H.2 is needed for the MP part when we want to apply week 
%convergence, see section 3)}

Given $(t,x)\in \R^+ \times \R$ and a family of densities $(p_t)_{t\leq T}$ we set
\begin{equation} \label{def:B}
B(t,x;p):= \int_0^t(K_{t-s}\ast p_s)(x) ds.
\end{equation}
We now define the notion of a weak solution to \eqref{NLSDEsimple}. 

\begin{mydef}
\label{def:weakSDE}
The family $(\Omega,\F,\P,(\F_t),X,W)$ is said to be a weak solution
to the equation \eqref{NLSDEsimple} up to time $T>0$ if:
\begin{enumerate}
\item $(\Omega,\F,\P,(\F_t))$ is a filtered probability space.
\item The process $X:=(X_t)_{t\in [0,T]}$ is real-valued,
continuous, and $(\F_t)$-adapted. In addition, the probability 
distribution of~$X_0$ has density~$p_0$.
\item The process $W:=(W_t)_ {t\in [0,T]}$ is a 
    one-dimensional $(\F_t)$-Brownian motion.
    \item The probability distribution $\P\circ X^{-1}$ has 
    time marginal densities $(p_t, \ t\in[0,T])$
   with respect to Lebesgue measure which satisfy
  \begin{equation}
\label{density_infty}
\forall 0<t\leq T,~~~\|p_t\|_{L^\infty(\R)}\leq \frac{C_T}{\sqrt{t}}.
\end{equation}
%\begin{equation}
%\label{density_space}
%\forall 0<t\leq T,~~~\|p_t\|_{L^2(\R)}\leq \frac{C_T}{t^{1/4}}.
%\end{equation}
\item For all $t\in [0,T]$ and $x\in \R$, one has that $\int_0^t|b(s,x)|~ds<\infty.$
 % fin de rouge
    \item $\P$-a.s. the pair $(X,W)$ satisfies \eqref{NLSDEsimple}.
\end{enumerate}
\end{mydef}
\begin{remark} \label{rk:cond-integr-densite}
For any $T>0$
% and $(s,x)\in (0,T]\times \R$. Setting
%$$I(\varepsilon):=\int_0^{s-\varepsilon} \int K_{s- \tau}(x-y) 
%q_\tau(y)dyd\tau,$$
%one has
%$$|I(\varepsilon)|\leq \int_0^{s-\varepsilon} \|q_\tau\|_{L^2(\R)} 
%\|K_{s-\tau}\|_{L^2(\R)} d\tau.$$
%Inequality~\eqref{density_space}
Inequality~\eqref{density_infty} and Hypothesis (H-4) lead to
$$ \sup_{0\leq t\leq T}\sup_{x\in\R}|B(t,x,p)| \leq C_T. $$
% allow one to let $\varepsilon \to 0$
%and get
%$$ \sup_{0\leq t\leq T} \int_0^t |B(s,x,p)|~ds<\infty. $$ 
\end{remark}

The following theorem provides existence and uniqueness of the weak 
solution to~\eqref{NLSDEsimple}.

\begin{theorem}
\label{mainTh}
Let $T>0$. Suppose that $p_0 \in L^1(\R)$ is a probability density 
function 
and $b\in L^\infty([0,T]\times \R)$ is continuous w.r.t. the space 
variable. Under 
the hypothesis $(\text{H})$, Eq.~\eqref{NLSDEsimple} admits a unique
weak solution in the sense of Definition \ref{def:weakSDE}. 
\end{theorem}
 We finally state an easy result which is useful to prove the propagation of chaos in the case of Keller-Segel kernel (see \cite{JTT}):\begin{cor}
\label{cor:forPS}
In addition to the assumptions of Theorem \ref{mainTh} suppose the following hypothesis: 
\begin{enumerate}[H-7.]
\item for any $t>0$, $K_t$ is in $L^2(\R)$ and the function $f_2(t):=\int_0^t \frac{\|K_{t-s}\|_{L^2(\R)}}{s^{1/4}} ds$ is well defined and bounded on~$[0,T]$.
\end{enumerate}
Then, there exists a unique weak solution to \eqref{NLSDEsimple} in the sense of the Definition \ref{def:weakSDE} modified as follows: Instead of \eqref{density_infty} one imposes
\begin{equation}
\label{density_space}
\forall 0<t\leq T,~~~\|p_t\|_{L^2(\R)}\leq \frac{C_T}{t^{1/4}}.
\end{equation}
\end{cor}

Our next result concerns the well-posedness of the 
the one-dimensional parabolic-parabolic Keller-Segel model
\begin{subequations}
\label{KS}
\begin{empheq}[left={}\empheqlbrace]{align}
  & \frac{\partial\rho}{\partial t}(t,x) =\frac{\partial}{\partial x} 
  \cdot (\frac{1}{2}\frac{\partial\rho}{\partial x} - \chi \rho
  \frac{\partial c}{\partial x})(t,x), \quad  t>0,~~x \in \R, 
  \label{KS1} 
  \\
  &  \frac{\partial c}{\partial t}(t,x) =  
  \frac{1}{2}\frac{\partial^2c}{\partial 
  x^2}(t,x) - \lambda c(t,x) + \rho(t,x), \quad t>0,~~x \in \R, 
  \label{KS2}\\
  &  \rho(0,x)= \rho_0(x),~~~c(0,x)= c_0(x). \nonumber
\end{empheq}
\end{subequations}
The parameters $\chi$ and $\lambda$ are strictly positive. As this 
system preserves the total mass, that is,
$$ \forall t>0,~~~\int_\Omega \rho(t,x)dx = \int_\Omega \rho_0(x)dx =: 
M, $$
%The pair $(\frac{\rho(t,x)}{M},\frac{c(t,x)}{M})$.
the new functions $\tilde{\rho}(t,x) := \frac{\rho(t,x)}{M}$ and
$\tilde{c}(t,x) := \frac{c(t,x)}{M}$
satisfy the system~\eqref{KS} with the new 
parameter~$\tilde{\chi} := \chi M$. Therefore, w.l.o.g. we may and do
thereafter assume that $M=1$.
% \rouge{NE FAUDRAIT-IL PAS LAISSER $\chi$ 
%POUR BIEN MONTRER QU'IL N'EXISTE PAS DE REGIME CRITIQUE EN 
%DIMENSION~1~?}
%\milica{Sure, we can put it in the definition of the KS kernel.}
%Normalizing $c(t,x)$ by $e^{-\lambda t}$ shows that we may also assume 
%w.l.o.g. that $\lambda=0$.
%\milica{I stay with $\lambda$ as the last section is written with it 
%and it appears everywhere.}
%Furthermore, we assume from now on that $\rho_0 \in L^\infty(\R)\cap 
%L^1(\R)$ and $ c_0 \in  W_p^{2-\frac{2}{p}}(\R)$, for a $p>3$.\\\\

Denote by $g_t$ the density of $W_t$. We define the notion 
of solution for the system \eqref{KS}:
\begin{mydef}
\label{notionOfSol}
Given the functions $\rho_0$ and $c_0$, and the constants $\chi >0$, 
$\lambda\geq0$, $T>0$, the pair $(\rho,c)$ is said 
to be a 
solution to~\eqref{KS} if~$\rho(t,\cdot)$ is a probability density 
function for 
every $0\leq t\leq T$,
$c$ is in $L^\infty([0,T];C_b^1(\R))$,
one has
%$\|\rho(t,\cdot)\|_{L^{2}(\R)}\leq \frac{C}{t^{\frac{1}{4}}}$ and
$\|\rho(t,\cdot)\|_{L^\infty(\R)}\leq \frac{C_T}{\sqrt{t}}$
for any $t\in(0,T]$,
and the following equality
\begin{equation} \label{eq:rho-KS}
\rho(t,x) =g_t\ast \rho_0(x) - \chi \int_0^t\frac{\partial g_{t-s}}
{\partial x}\ast(\frac{\partial c}{\partial x} (s,\cdot)
~\rho(s,\cdot))(x)~ds
\end{equation}
is satisfied in the sense of the distributions with
\begin{equation} \label{eq:c-KS}
c(t,x) = e^{-\lambda t}(g(t,\cdot \ )\ast c_0)(x)+ \int_0^t 
e^{-\lambda s} (g_s \ast \rho(t-s,\cdot))(x)~ds.
\end{equation}

\end{mydef}

Notice that the function $c(t,x)$ defined by~\eqref{eq:c-KS}
is a mild solution to~\eqref{KS2}.
These solutions are known as integral solutions and they have 
already been studied in PDE literature for the two-dimensional 
Keller-Segel model for which sub-critical and critical regimes exist
depending on the parameters of the model
(see~\cite{Corrias2014} and references therein). 
In the one-dimensional case there is no critical regime as shown by the 
following theorem.
\begin{cor}
\label{KSmain} 
Assume that $\rho_0 \in L^1(\R)$ and $ c_0 \in  C_b^1(\R)$. 
Given any $\chi>0$, $\lambda\geq0$ and $T>0$,
the time marginals~$\rho(t,x)\equiv p_t(x)$
of the probability 
distribution of the unique solution to Eq.~\eqref{NLSDEd} with $d=1$
and the corresponding function $c(t,x)$
provide a global solution to \eqref{KS} in the sense of 
Definition \ref{notionOfSol}. Any other solution $(\rho^1,c^1)$ with 
the same initial condition $(\rho_0,c_0)$ satisfies
$\|\rho^1(t,\cdot)-\rho(t,\cdot)\|_{L^1(\R)}=0$
and $\|\frac{\partial c^1}{\partial x}(t,\cdot)-\frac{\partial 
c}{\partial 
x}(t,\cdot)\|_{L^1(\R)}=0$ for every $0\leq t\leq T$.
\end{cor}

\begin{remark}
%The condition on $c_0$ can be replaced by any other condition that 
%ensures that \newline $\frac{\partial}{\partial x}e^{-\lambda t} 
%\E(c_0(x+W_t))$ is bounded by a deterministic constant. In addition, 
From estimates below we could deduce some additional regularity 
results which we do not need here: See 
Remark~\ref{rk:regularity-result}. In particular, if $\rho_0\in 
L^{\infty}(\R)$, then $\rho \in L^\infty([0,T];L^1\cap L^\infty(\R))$.
If $\rho_0\in L^{2}(\R)$, then $\rho \in L^\infty([0,T];L^1
\cap L^2(\R))$ 
and $  t^{1/4}\|\rho_t\|_{L^{\infty}(\R)}\leq C$.
%; in addition, one can then easily find modifications of the hypotheses on the kernel~$K$ allowing to get unique weak solutions with constraints on $\|p_t\|_{L^2(\R)}$.
 As explained in the introduction, we prefer to only suppose that $\rho_0 \in L^1(\R)$. 
\end{remark}
%Our last main result shows the 
%relationship between the non-linear 
%SDE~\eqref{NLSDEsimple} and the Keller-Segel system~\eqref{KS}.

%\begin{cor} \label{theo:stoch-interpr-KS}
%Assume that $\rho_0 \in L^1(\R)$ and $ c_0 \in  C_b^1(\R)$. 
%%The kernel
%%$ K^\sharp_t(x):=\chi e^{-\lambda t}
%%\frac{-x}{2\pi t^{3/2}}e^{-\frac{x^2}{2t}} $
%% and the function $b(t,x) :=\chi e^{-\lambda t} \E c'_0(x+W_t)$ 
%%satisfy 
%% the hypothesis of Theorem \ref{mainTh}.
%%
%% is b
%%Assume that 
%Denote by $\rho_t\equiv p_t$ the time marginals of the probability 
%distribution of the unique solution to Eq.~\eqref{NLSDEd} with 
%$d=1$.
%Set
%$$ c_t := e^{-\lambda t}g(t,\cdot)\ast c_0+ \int_0^t e^{-\lambda s} 
%\rho_{t-s}\ast g(s,\cdot) ds. $$
%Then the pair $(\rho_t,c_t)$ is  the unique global solution to 
%the Keller-Segel system~\eqref{KS} in the sense of Definition 
%\ref{notionOfSol}.
%\end{cor}

\section{Preliminary: A density estimate}
\label{sec:prel}
\label{genInftyNorm}
In the sequel, we will get local solutions to \eqref {NLSDEsimple} and extend them to global solutions by means of an iterative procedure. The $L^\infty$-norms of the successive drifts are needed to be bounded from above uniformly w.r.t. the iteration step. Standard density estimates obtained by using Girsanov theorem or PDE analysis do not help to this purpose. The reason is  that they involve constants which exponentially depend on the $L^\infty$-norm (or even H\"older-norm) of the drifts. We therefore proceed by using an accurate pointwise estimate (with explicit constants) on densities of one-dimensional diffusions with bounded measurable
drifts. Estimate~\eqref{QianZhengUS} below is obtained by using a stochastic technique. Its drawback is that the map $y\mapsto p_y^\beta(t,x,y)$ is not a probability density function. However, it suffices to nicely bound the successive drifts of the Picard iterations as shown by Proposition \ref{iteratesRes}.

Let $X^{(b)}$ be a process defined by
\begin{equation}
\label{ProcessTimeDrift}
X_t^{(b)} = X_0 + \int_0^t b(s,X_s^{(b)})~ds + W_t, \quad t\in [0,T].
\end{equation}
To obtain $L^\infty(\R)$ estimates for the transition probability 
density $p^{(b)}(t,x,y)$ of $X^{(b)}$ under 
the only assumption that the drift $b(t,x)$ is measurable and uniformly 
bounded we slightly extend the estimate proved 
in Qian and Zheng~\cite{QianZheng} for time homogeneous drift coefficients $b(x)$. 
We here propose a proof different from the original one. It avoids the use of densities of pinned 
diffusions and the claim that $p^{(b)}(t,x,y)$ is continuous w.r.t. 
all the variables which does not seem obvious to us.
In our proof we adapt the method in~\cite{Makhlouf}, 
the main difference being that 
instead of the Wiener measure our 
reference measure is the probability distribution of the particular 
diffusion process $X^\beta$ considered in~\cite{QianZheng} and defined 
by
$$ X^\beta_t = X_0 + \beta \int_0^t \text{sgn}(y-X^\beta_s)~ds + W_t. $$
%In the sequel we denote by $p^\beta_y(t,x,z)$ the transition density 
%of $X^\beta$ and use Qian and Zheng's explicit representation for it
%(see~\cite{QianZheng} and our appendix section~\ref{sec:appendix}).

\begin{theorem}
\label{QianZhengUSTh}
Let $X^{(b)}$ be the process defined in \eqref{ProcessTimeDrift} with $X_0=x$. Let $p^\beta_y(t,x,z)$ be the transition density 
of $X^\beta$.
Assume $\beta:=\sup_{t\in[0,T]} \|b(t,\cdot)\|_\infty <\infty$.  Then 
for all $ y \in \R$  and $t\in(0,T]$ it holds that
\begin{equation} \label{QianZhengUS}
p^{(b)}(t,x,y) \leq p^\beta_y(t,x,y)= \frac{1}{\sqrt{2\pi t}} 
\int_{\frac{|x-y|}{\sqrt{t}}}^{\infty}  z e^{-\frac{(z-\beta 
\sqrt{t})^2}{2}} dz.
\end{equation}
\end{theorem}

\begin{proof}
%For a function $\varphi(\theta, \xi, z)$ defined on $[0,T]\times \R \times \R$ and a real valued random variable $X$, we denote
%$$\frac{\partial}{\partial \xi}\varphi (\theta, X, z):= \frac{\partial}{\partial \xi}\varphi (\theta, \xi, z)|_{\xi= X}.$$
Let $f \in \mathcal{C}_K^\infty(\R)$ and fix $t\in (0,T]$. Consider the 
parabolic PDE driven by the infinitesimal generator of~$X^\beta$:
\begin{equation}
\begin{cases}
\label{sysSgn}
& \frac{\partial u}{\partial t}(s,x) + \frac{1}{2} \frac{\partial^2 
u}{\partial x ^2}(s,x) + \beta \text{sgn}(y-x) \frac{\partial 
u}{\partial x}(s,x) =0 , \quad 0\leq s< t,~~~x\in \R, \\
&u(t,x)= f(x),~~~x\in \R.\\
\end{cases}
\end{equation}
In view of Veretennikov \cite[Thm.1]{Veretennikov1982}
there exists a solution 
$u(s,x)\in W^{1,2}_p([0,t]\times R)$. 
Applying the It\^o-Krylov formula to $u(s,X^\beta_s)$ we obtain that
$$u(s,x)= \int f(z) p^\beta_y(t-s,x,z)~dz. $$
The formula~(\ref{eq:p-beta-x-y-z}) from our appendix 
%for some constant $C$ depending on $\beta$, $T$ and $y$ only we thus 
%have, $|\frac{\partial p^\beta_y}{\partial x}(t-s,x,z)| \leq \frac{C}{t-s}$.
%\milica{In addition, Corollary~\ref{cor:ineq-derivee-densite} in the appendix 
%allows us to }
%This
 allows us to differentiate under the integral sign:
$$ \frac{\partial u}{\partial x}(s,x) 
= \int f(z) \frac{\partial p^\beta_y}{\partial x}(t-s,x,z)~dz,
~~~\forall 0\leq s< t\leq T. $$
 % fin de rouge

Fix $0<\varepsilon<t$. Now apply the It\^o-Krylov formula to 
$u(s,X^{(b)}_s)$ for $0\leq s\leq t-\varepsilon$ and use the PDE~ 
\eqref{sysSgn}. It comes:
%~\footnote{Below we use the following notation which
%takes care of the fact that $X_s^{(b)}$ depends on $x$:
%$\left(\frac{\partial u}{\partial x}\right)(s,X^{(b)}_s) := 
%\frac{\partial}{\partial \xi}u(s,\xi)|_{\xi= X_s^{(b)}}$ and
%$\left(\frac{\partial}{\partial x}p^\beta_y\right)(t-s,X_s^{(b)},z) := 
%\frac{\partial}{\partial \xi}p^\beta_y(t-s,\xi,z)_{\xi= X_s^{(b)}}$.}:
$$\E( u(t-\varepsilon, X^{(b)}_{t-\varepsilon}))= u(0,x) + \E 
\int_0^{t-\varepsilon}(b(s,X^{(b)}_s)-\beta 
\text{sgn}(y-X_s^{(b)}))\frac{\partial u}{\partial 
x}(s,X^{(b)}_s)~ds. 
$$

%Now, by Girsanov's theorem,
%for some constant $C$ depending on $T$ and~$\beta$ only we have
%\begin{equation*} 
%%\label{lemmaDerOfDensBound:girsanov}
%\E \left|\frac{\partial}{\partial x}p^\beta_y(t-s,X_s^{(b)},z) \right|
%\leq C \sqrt{\E \left|\frac{\partial}
%{\partial x}p^\beta_y(t-s,W_s^x,z) \right|^2}.
%\end{equation*}
In view of~Corollary~\ref{cor:ineq-derivee-densite} in the appendix  
there exists a function $h \in L^1([0,t]\times \R)$ such that
\begin{equation} \label{derOfDensBound}
\forall 0<s<t\leq T,~\forall y,z \in \R,
~~\E \left|\frac{\partial p^\beta_y}{\partial 
x}(t-s,X_s^{(b)},z)
%\footnotemark[\value{footnote}]
\right|
\leq C_{T,\beta,x,y} h(s,z).
\end{equation}
Consequently,
\begin{equation*}
\begin{split}
\E( u(t-\varepsilon, X^{(b)}_{t-\varepsilon}))
&= \int f(z) p^\beta_y(t,x,z)~dz  \\
&\qquad + \int f(z) \int_0^{t-\varepsilon} \E \left\{
(b(s,X_s^{(b)})-\beta \text{sgn}(y-X_s^{(b)}))\frac{\partial 
p^\beta_y}{\partial x}(t-s,X_s^{(b)},z)\right\}~ds~dz.
\end{split}
\end{equation*}
Let now $\epsilon$ tend to~0. By Lebesgue's dominated convergence 
theorem we obtain
\begin{equation*}
\begin{split}
\int f(z) p^{(b)}(t,x,z)dz &= \int f(z) p^\beta_y(t,x,z)dz  \\
&\qquad + \int f(z) \int_0^{t} \E \left\{ (b(s,X_s^{(b)})-\beta 
\text{sgn}(y-X_s^{(b)}))\frac{\partial p^\beta_y}{\partial 
x}(t-s,X_s^{(b)},z) \right\} ~ds~dz.
\end{split}
\end{equation*}
Therefore the density $p^{(b)}$ satisfies:
$$ p^{(b)}(t,x,z)= p^\beta_y(t,x,z) + \int_0^t \E 
\left\{ (b(s,X_s^{(b)})-\beta 
\text{sgn}(y-X_s^{(b)}))\frac{\partial p^\beta_y}{\partial 
x}(t-s,X_s^{(b)},z) \right\}~ds. $$
As noticed in~\cite{QianZheng}, in view of Formula~(\ref{eq:p-beta}) 
from our appendix we have for any $x \in \R$
$$ (b(s,x)-\beta \text{sgn}(y-x))
\frac{\partial}{\partial x}p^\beta_y(t-s,x,y) \leq 0.$$
This leads us to choose $z=y$ in the preceding equality, which gives us
$$p^{(b)}(t,x,y)= p^\beta_y(t,x,y) + \int_0^t \E \left\{
(b(s,X_s^{(b)})-\beta \text{sgn}(y-X_s^{(b)}))
\frac{\partial p^\beta_y}{\partial 
x}(t-s,X_s^{(b)},y)\right\}~ds, $$
from which
$$ \forall t\leq T,~~p^{(b)}(t,x,y)\leq p^\beta_y(t,x,y). $$
We finally use Qian and Zheng's explicit representation (see~\cite{QianZheng} and our appendix section~\ref{sec:appendix}).

\end{proof}

\begin{cor}
\label{prel:densityEST}
Assume $X_0$ is distributed according to the probability density 
function $p_0$ on $\R$. Denote by $p(t,\cdot)$ the probability density of 
$X_t^{(b)}$.
%\begin{enumerate}
%\item If $p_0 \in L^{\infty}(\R)$, then for any $t\in(0,T]$ we have 
%\begin{equation}
%\label{densityInfty2}
%\|p(t,\cdot)\|_\infty\leq  2\|p_0\|_\infty+\beta.
%\end{equation}
%\item If  $p_0 \in L^{p}(\R)$ for some $p\geq 1$, then for any 
%$t\in(0,T]$ we have
%\begin{equation}
%\label{densityLp}
%\|p(t,\cdot)\|_\infty\leq  
%\frac{C\|p_0\|_{L^{p}(\R)}}{t^{\frac{1}{2p}}}+\beta.
%\end{equation}
%\end{enumerate}
One has 
\begin{equation}
\label{densityLp}
\|p(t,\cdot)\|_{L^\infty(\R)} \leq  
\frac{1}{\sqrt{2\pi t}}+\beta.
\end{equation}
\end{cor}

\begin{proof}
%We express this density at time $t$ and point $y$ through the transition density $p^{(b)}(t,x,y)$:
%$$p(t,y)=\int p_0(x) p^{(b)}(t,x,y)dx.$$
In view of \eqref{QianZhengUS} we have
\begin{equation*}
\begin{split}
p(t,y) &\leq \frac{1}{\sqrt{2\pi t}}\int p_0(x) 
\int_{\frac{|x-y|}{\sqrt{t}}}^{\infty}  z e^{-\frac{(z-\beta 
\sqrt{t})^2}{2}} dzdx \\
&\leq \frac{1}{\sqrt{2\pi t}} \int p_0(x) \int_{\frac{|x-y|}{\sqrt{t}} 
-\beta \sqrt{t}}^{\infty}  (z+\beta \sqrt{t}) e^{-\frac{z^2}{2}} dzdx\\
    &=\frac{1}{\sqrt{2\pi t}} (\int p_0(x)e^{-\frac{(|x-y|-\beta 
    t)^2}{2t}}dx + \beta \sqrt{t} \int p_0(x) 
    \int_{\frac{|x-y|}{\sqrt{t}} -\beta 
    \sqrt{t}}^{\infty}e^{-\frac{z^2}{2}} dz dx )\\
    &\leq \frac{1}{\sqrt{2\pi t}} \int p_0(x) e^{-\frac{(|y-x|-\beta 
    t)^2}{2t}}dx + \beta.
\end{split}
\end{equation*}
\end{proof}

\begin{remark} \label{rk:regularity-result}
If $p_0 \in L^{\infty}(\R)$, the above calculation shows that
$$ \|p(t,\cdot)\|_{L^\infty(\R)}\leq  2\|p_0\|_{L^\infty(\R)}+\beta. $$
% If  $p_0 \in L^{1}(\R)$, then
%$$\frac{1}{\sqrt{2\pi t}} \int p_0(x) e^{-\frac{(|y-x|-\beta 
%t)^2}{2t}}dx\leq \frac{\|p_0\|_{L^1(\R)}}{\sqrt{2\pi t}}.$$
If  $p_0 \in L^{p}(\R), p>1$, H\"older's inequality leads to
$$\frac{1}{\sqrt{2\pi t}} \int p_0(x) e^{-\frac{(|y-x|-\beta 
t)^2}{2t}}dx\leq\frac{\|p_0\|_{L^{p}(\R)}}{\sqrt{2\pi t}} (\int  
e^{-q\frac{(|y-x|-\beta t)^2}{2t}}dx)^{1/q}\leq \frac{C_q 
t^{\frac{1}{2q}}}{\sqrt{t}}= \frac{C_q}{t^{\frac{1}{2p}}}.$$
%As above, $\int  e^{-q\frac{(|y-x|-\beta t)^2}{2t}}dx\leq 2 C_q 
%\sqrt{t}.$ Thus, the estimate in \eqref{densityLp} is satisfied as 
%well.
\end{remark}

\section{A non-linear McKean--Vlasov--Fokker--Planck equation}
\label{subs:uniq}
%The idea is to characterize the one dimensional distributions of a 
%solution that has the properties as in Theorem \ref{localEX}. We 
\begin{proposition}
\label{MeqUniq}
Let $T>0$. Assume $p_0 \in L^1(\R)$, $b\in L^\infty([0,T]\times \R)$ 
and Hypothesis $(\text{H})$. Let $(\Omega,\F,\P,(\F_t),X,W)$ be a weak 
solution to 
\eqref{NLSDEsimple} until $T$.
%that belongs to the class of probability measures on 
%$\mathcal{C}([0,T];\R)$ 
%whose one dimensional marginals are densities which for any 
%$t\in(0,T)$ 
%satisfy \eqref{density_space}. 
Then,
\begin{enumerate}
\item The marginals  $(p_t)_{t\in [0,T]}$ satisfy in the sense 
of the distributions the mild equation
 \begin{equation}
 \label{mildEQ}
 \forall t\in(0,T], \quad p_t= g_t\ast p_0-\int_0^t\frac{\partial g_{t-s}}{\partial x} \ast(p_s(b(s,\cdot)+ B(s,\cdot \ ;p))ds.
 \end{equation}
 \item Equation \eqref{mildEQ} admits at most one solution $(p_t)_{ 
 t\in [0,T]}$  which for any $t\in[0,T]$ belongs to $L^1(\R)$ and 
 satisfies \eqref{density_infty}.
 %\eqref{density_space} and
\end{enumerate}
\end{proposition}
\begin{proof}
We successively prove~\eqref{mildEQ} and the uniqueness of its solution 
in $L^1(\R)$.
\paragraph*{1.} 
%Notice that, in view of Hypothesis~(H), for all $t\in(0,T]$ and $x\in 
% \R$ one has $|B(t,x;p)|\leq C_T$.

Now, for $f\in C^2_b(\R)$ consider the Cauchy problem
 \begin{equation}
 \label{pdeNOdrift}
 \begin{cases}
 & \frac{\partial G}{\partial s} + \frac{1}{2} \frac{\partial^2 
 G}{\partial x^2}  = 0, \quad 0\leq s< t,~~~x\in \R,\\
 & \lim_{s\to t^-}G(s,x)= f(x). 
 \end{cases}
 \end{equation}
 The function
 $$G_{t,f}(s,x)=\int f(y) g_{t-s}(x-y)dy$$
 is a smooth solution to \eqref{pdeNOdrift}. 
% On the other hand, we can see it as
% $$G_{t,f}(s,y)=\E(f(Y_t^{s,y})),$$
% where $dY_\tau^{s,y}=dW_\tau, \tau\leq t$ and $Y_s^{s,y}=y$.\\\\
Applying It\^o's formula we get
%to $G_{t,f}(s,X_s)$ between $0$ and $t$:
\begin{equation*}
 \begin{split}
  G_{t,f}(t,X_t)- G_{t,f}(0,X_0) &= \int_0^t
     \frac{\partial G_{t,f}}{\partial s}(s,X_s)ds + \int_0^t 
     \frac{\partial G_{t,f}}{\partial x}(s,X_s) (b(s,X_s)+B(s,X_s;p))ds 
     \\
     &~~~+ \int_0^t \frac{\partial G_{t,f}}{\partial x} (s,X_s) dWs + 
     \frac{1}{2}\int_0^t \frac{\partial^2 G_{t,f}}{\partial x^2} 
     (s,X_s)ds.
 \end{split}
 \end{equation*}
 Using \eqref{pdeNOdrift} we obtain
 \begin{equation}
 \label{derivingME}
 \E f(X_t)= \E G_{t,f}(0,X_0) + \int_0^t \E\left[\frac{\partial 
 G_{t,f}}{\partial 
 x}(s,X_s) (b(s,X_s)+B(s,X_s;p))\right]ds=: I+II.
 \end{equation}
 On the one hand one has
 $$I= \int \int f(y) g_t(y-x)dy~p_0(x)dx=\int f(y) (g_t \ast p_0)(y) dy.$$
 On the second hand one has
 \begin{align*}
     & II= \int_0^t \int \frac{\partial}{\partial x} \big [ \int f(y) g_{t-s}(x-y)dy \big ] (b(s,x)+B(s,x;p))p_s(x) dxds\\
     & =\int_0^t \int \int f(y) \frac{\partial g_{t-s}}{\partial x} 
     (x-y)dy (b(s,x)+B(s,x;p))p_s(x) dxds\\
     &=- \int f(y) \int_0^t [\frac{\partial g_{t-s}}{\partial x} 
     \ast( (b(s,\cdot)+B(s,\cdot;p))p_s)](y)ds dy.
 \end{align*}
 Thus \eqref{derivingME} can be written as
$$\int f(y)p_t(y)dx = \int f(y) (g_t\ast p_0)(y) dy+ \int f(y) 
\int_0^t[ \frac{\partial g_{t-s}}{\partial x} \ast( 
(b(s,\cdot)+B(s,\cdot;p))p_s)](y)ds dy, $$
 which is the mild equation \eqref{mildEQ}.\\\\
 
\paragraph*{2.} Assume $p_t^1$ and $p_t^2$ are two mild solutions
in the sense of the distributions to \eqref{mildEQ} which 
satisfy
 $$\exists C>0, \forall t\in (0,T], ~~~\|p_t^1\|_{L^\infty(\R)} + 
 \|p_t^2\|_{L^\infty(\R)}\leq \frac{C_T}{\sqrt{t}}. $$
%In view of Hypothesis (H.4) we have
% $$\sup_{t\in [0,T]}\|B(t,x;p^1)\|_{L^\infty(\R)} + \sup_{t\in 
% [0,T_0]}\|B(t,x;p^2)\|_{L^\infty(\R)}\leq C_T.$$
 
%  After a Cauchy-Schwarz inequality in the time integral in 
%  \eqref{mildEQ}, one gets
% $$\exists C>0, \forall t\in (0,T], ~~~\|p_t^1\|_\infty + 
% \|p_t^2\|_\infty\leq \frac{C_T}{\sqrt{t}}.$$
% Here we used that $g(t,\cdot)\leq \frac{C_T}{\sqrt{t}}$ and 
% $\|\frac{\partial}{\partial y} g(t-s,\cdot \ )\|_{L^{2}(\R)} \leq 
% \frac{C}{(t-s)^{1/4}}$.
% 
% \rouge{ATTENTION. L'EQ. MILD N'EST PAS SATIFAITE PONCTUELLEMENT. Thus, 
% to justify the calculations below we probably need to use 
 %that if $F\in L^1(\R)$ and $F=0$ in the sense of the distributions,
 %then $F=0$ almost everywhere (exercice. Hint: approximate $F$ by 
 %convolutions.}
Then, 
 \begin{equation*}
 \begin{split}
      \|p_t^1- p_t^2\|_{L^1(\R)} &\leq \int_0^t \|\frac{\partial 
      g_{t-s}}{\partial x}\ast[ 
      B(s,\cdot \ ;p^1)p^1_s-B(s,\cdot \ ;p^2)p^2_s)\|_{L^1(\R)}ds \\
      &\quad + \int_0^t \|\frac{\partial g_{t-s}}{\partial x} 
      \ast[b(s,\cdot \ )(p^1_s-p^2_s)]\|_{L^1(\R)}  ds\\
      &\leq \int_0^t \|\frac{\partial g_{t-s}}{\partial x} 
      \ast[(B(s,\cdot;p^1)-B(s,\cdot;p^2))p^1_s]\|_{L^1(\R)}ds \\
      &\quad +\int_0^t\|\frac{\partial g_{t-s}}{\partial x} 
      \ast[(p^1_s-p^2_s)B(s,\cdot \ ;p^2)]\|_{L^1(\R)}ds\\ 
      &\quad +\int_0^t \|\frac{\partial g_{t-s}}{\partial x} 
      \ast[b(s,\cdot \ )(p^1_s-p^2_s)]\|_{L^1(\R)} ds \\
      &=:I+II+III.
  \end{split}
  \end{equation*}
   As
  $$\|\frac{\partial g_{t-s}}{\partial x}\|_{L^1(\R)} \leq 
  \frac{C_T}{\sqrt{t-s}},$$ 
  the convolution inequality $\|f\ast h\|_{L^1(\R)}\leq 
  \|f\|_{L^1(\R)}\|h\|_{L^1(\R)}$ and 
  Remark~\ref{rk:cond-integr-densite} lead to
  \begin{align*}
      &II\leq \int_0^t \|\frac{\partial g_{t-s}}{\partial x}\|_{L^1(\R)} \|(p^1_s-p^2_s)B(s,\cdot;p^2)\|_{L^1(\R)}ds\leq C_T 
      \int_0^t \frac{\|p^1_s-p^2_s\|_{L^1(\R)}}{\sqrt{t-s}}ds.\\
  \end{align*}
  As $b$ is bounded, we also have
  $$|III|\leq C_T \int_0^t 
  \frac{\|p^1_s-p^2_s\|_{L^1(\R)}}{\sqrt{t-s}}ds.$$
  We now turn to $I$. Notice that
%$$B(s,\cdot;p^1)-B(s,\cdot;p^2) =\int_0^s (K_{s-\tau} \ast 
%(p_\tau^1-p_\tau^2))(\cdot)d\tau,$$
%from which
$$\|B(s,\cdot;p^1)-B(s,\cdot;p^2)\|_{L^1(\R)} \leq\int_0^s \|K_{s-\tau}\|_{L^1(\R)} 
\|p_\tau^1-p_\tau^2\|_{L^1(\R)}d\tau, $$
  %Applying the convolution inequality, bound on $p_s^1$ and the previous computations,
from which, since by hypothesis $(p_t)$ satisfies \eqref{density_infty},
\begin{equation*}
\begin{split}
I &\leq \int_0^t \frac{C_T}{\sqrt{t-s} \sqrt{s}}\int_0^s 
  \|K_{s-\tau}\|_{L^1(\R)} \|p_\tau^1-p_\tau^2\|_{L^1(\R)}d\tau ds \\
&= \int_0^t\|p_\tau^1-p_\tau^2\|_{L^1(\R)} \int_\tau^t 
  \frac{C_T}{\sqrt{t-s} \sqrt{s}}\|K_{s-\tau}\|_{L^1(\R)} ds d\tau.
\end{split}
\end{equation*}  
  In addition, using Hypothesis (H-4),
  $$\int_\tau^t \frac{1}{\sqrt{t-s} \sqrt{s}}\|K_{s-\tau}\|_{L^1(\R)} 
  ds \leq \frac{1}{\sqrt{\tau}}\int_\tau^t \frac{1}{\sqrt{t-s} 
  }\|K_{s-\tau}\|_{L^1(\R)} ds=  \frac{1}{\sqrt{\tau}}\int_0^{t-\tau} 
  \frac{\|K_{s}\|_{L^1(\R)}}{\sqrt{t-\tau-s} }ds  \leq 
  \frac{C_T}{\sqrt{\tau}}.$$
  It comes:
  $$I\leq C_T \int_0^t \frac{\|p_\tau^1-p_\tau^2\|_{L^1(\R)} 
  }{\sqrt{\tau}}d\tau.$$
 Gathering the preceding estimates we obtain
  $$\|p_t^1- p_t^2\|_{L^1(\R)} \leq C_T \int_0^t 
  \frac{\|p_s^1-p_s^2\|_{L^1(\R)}}{\sqrt{t-s}}ds +  C_T \int_0^t 
  \frac{\|p_s^1-p_s^2\|_{L^1(\R)} }{\sqrt{s}}ds .$$
Applying a  Singular Gronwall Lemma (see Lemma~\ref{ModifiedSgGronwall} 
below), we conclude 
  $$\forall t\in (0,T], \quad \|p_t^1-p_t^2\|_{L^1(\R)}=0,$$
  which ends the proof.
%  This implies that $p^1_t$ and $p^2_t$ are equal almost everywhere 
%for every $t\in (0,T]$.
\end{proof}

In the above proof we have used the following result:
\begin{lemma}
\label{ModifiedSgGronwall}
Let $(u(t))_{t\geq 0}$ be a non-negative bounded function such that for 
a given $T>0$, there exists a positive constant $C_T$ such that for any 
$t \in (0,T]$:
\begin{equation}
\label{Gronwall:generalExpression}
u(t)\leq C_T \int_0^t \frac{u(s)}{\sqrt{s}} ds + C_T \int_0^t \frac{u(s)}{\sqrt{t-s}} ds.
\end{equation}
Then, $u(t)=0$ for any $t \in (0,T]$.
\end{lemma}
\begin{proof}
Set $u_t^\ast := \sup_{s\leq t} u(s)$. The inequality 
\eqref{Gronwall:generalExpression} implies that
$$u_t^\ast \leq 4C_T \sqrt{t} u_t^\ast.$$
Set $T^\ast:= \frac{1}{32 C_T^2}$. If $T\leq T^\ast$, then for $t\leq 
T$, we have $u_t^\ast \leq \frac{u_t^\ast}{\sqrt{2}}$. Thus, 
$u_t^\ast=0$ for every $t\leq T$ and the lemma is proved. If $T> 
T^\ast$, for $T^\ast<t< T $,
$$u(t)\leq C_T 
\int_0^{T^\ast}\left(\frac{1}{\sqrt{t-s}}+\frac{1}{\sqrt{s}}\right)u(s)ds
 + \int_{T^\ast}^t 
\left(\frac{1}{\sqrt{t-s}}+\frac{1}{\sqrt{s}}\right)u(s)ds.$$
 The first integral is null, since $u(T^\ast)=0$. Thus,
 %In the second, we apply the change of variables, $s-T^\ast = \theta$. 
 %Then,
 $$u(t)\leq \int_0^{t-T^\ast} \left( \frac{1}{\sqrt{t-T^\ast-\theta}} + 
 \frac{1}{\sqrt{T^\ast+\theta}} \right)u(T^\ast+\theta)d\theta\leq 
 \int_0^{t-T^\ast} \left( \frac{1}{\sqrt{t-T^\ast-\theta}} + 
 \frac{1}{\sqrt{T^\ast}} \right)u(T^\ast+\theta)d\theta.$$
 For $0<s\leq T-T^\ast$, define $v(s):= u(s+T^\ast)$. The previous 
 inequality becomes:
 $$v(t-T^\ast) \leq C_T \int_0^{t-T^\ast}\left( 
 \frac{1}{\sqrt{t-T^\ast-\theta}} + \frac{1}{\sqrt{T^\ast}} 
 \right)v(\theta)d\theta
 \leq \frac{\sqrt{T^\ast}+\sqrt{T-T^\ast}}{\sqrt{T^\ast}} 
 \int_0^{t-T^\ast} \frac{v(\theta)}{\sqrt{t-T^\ast-\theta}}d\theta.$$
 Setting $t-T^\ast=: \tau$, we thus have
 $$ v(\tau)\leq \frac{\sqrt{T^\ast}+\sqrt{T-T^\ast}}{\sqrt{T^\ast}} 
 \int_0^\tau \frac{v(\theta)}{\sqrt{\tau -\theta}}d\theta,
 ~~~\forall 0\leq\tau< T- T^\ast. $$
 Now we are in a position to apply a standard singular Gronwall lemma 
 (see 
 \cite[Lem. 7.1.1]{DanHenry}) and conclude that $v(\tau)=0$ for 
 $0\leq\tau \leq T-T^\ast$. Thus, $u(t)=0$ for $T^\ast\leq t\leq T$.
\end{proof}

\section{A local existence and uniqueness result for Equation \eqref{NLSDEsimple}}
\label{sec:local}
Set
\begin{equation}
\label{def:D(T)}
D(T):= \int_0^T \int_\R |K_t(x)|dx dt< \infty.
\end{equation}
%\milica{Delete?
%$(H.4)$ implies that for any $t\in (0,T)$
%$$\int_0^t \frac{\|K_{t-s}\|_{L^1(\R)}}{\sqrt{s}} ds+ \int_0^t 
%\frac{\|K_{t-s}\|_{L^2(\R)}}{s^{1/4}} ds\leq C_T.$$}
The main result in this section is the following theorem.
\begin{theorem}
\label{localEX}
Let $T_0>0$ be such that $D(T_0)<1$. Assume $p_0 \in L^1(\R)$ and $b\in 
L^\infty((0,T_0)\times \R)$ continuous w.r.t. space variable. Under 
Hypothesis~$(\text{H})$, Equation \eqref{NLSDEsimple} admits a unique
weak solution up to $T_0$ such that the probability distributions $\P \circ X_t^{-1}$ admit densities which satisfy \eqref{density_infty}.
% a unique 
%weak solution up to $T_0$. In addition, the 
%inequality~\eqref{density_infty} holds true.
\end{theorem}
\paragraph{Iterative procedure.} Consider the 
following sequence of SDE's. For $k=1$
\begin{equation}
\label{iter1}
\begin{cases}
& dX_t^1= b(t,X_t^1)~dt +\Big\{\int_0^t (K_{t-s}\ast 
p_0)(X_t^1)ds\Big\} dt+dW_t, \\
& X_0^1 \sim p_0.
\end{cases}
\end{equation}
Denote the drift of this equation by $b^1(t,x)$. Supposing that, in the 
step $k-1$, the one dimensional time marginals of the law of the 
solution have densities $(p^{k-1}_t)_{t\geq 0}$, we define the drift in the step 
$k$ as
$$b^k(t,x,p^{k-1})=b(t,x) + B(t,x;p^{k-1}).$$
The corresponding SDE is
\begin{equation}
\label{iterk}
\begin{cases}
& dX_t^k=b^k(t,X_t^k,p^{k-1})dt+ dW_t,  \\
& X_0^k \sim p_0.
\end{cases}
\end{equation}
 
%Then, we would construct a solution to \eqref{NLSDEsimple} as a weak limit of $\P^k$ as $k\to \infty$.
%\subsection{Local existence }
In order to prove the desired local existence and uniqueness result we 
set up the non-linear martingale problem related to \eqref{NLSDEsimple}.
\begin{mydef}
\label{defMP}
A probability measure $\Q$ on the canonical space 
$\mathcal{C}([0,T_0];\R)$ 
equipped with its canonical filtration and a canonical process $(w_t)$ is a solution to the non-linear 
martingale problem $(MP(p_0,T_0,b))$ if:
\begin{enumerate}[(i)]
\item $\Q_0 = p_0$.
\item For any $t\in (0,T_0]$, the one dimensional time marginals of $\Q$, denoted by $\Q_t$, have densities $q_t$ w.r.t. Lebesgue measure on $\R$. In addition, they satisfy 
%$q_t\in L^2(\R)$ and 
\begin{equation}
\label{defMP:linfnorm}
\forall 0<t\leq T_0, \quad \|q_t\|_{L^\infty(\R)} \leq \frac{C_{T_0}}{\sqrt{t}}.
%\forall 0<t\leq T_0, \quad \|q_t\|_{L^2(\R)} \leq \frac{C_{T_0}}{t^{1/4}}.
\end{equation}
\item For any $f \in C_K^2(\R)$ the process $(M_t)_{t\leq T_0}$, 
defined as
$$M_t:=f(w_t)-f(w_0)-\int_0^t \big[\frac{1}{2} \frac{\partial^2f}{\partial x^2}(w_u)+\frac{\partial f}{\partial x}(w_u)(b(u,w_u)+ \int_0^u \int K_{u- \tau}(w_u-y) q_\tau(y)dyd\tau\big)]du $$
is a $\Q$-martingale.
\end{enumerate}
\end{mydef}
%\noindent As a similar martingale problem will repeat more than once 
%later on, we introduced the notation $(MP(p_0,T_0,b))$.
%\noindent By theory of martingale problems, we have the following remark.
%\begin{remark}
%\label{MP-SDE}
%Proving that there exists a weak solution to \eqref{NLSDEsimple} is equivalent to proving the existence of a solution to $(MP(p_0,T_0,b))$ (ref).
%\end{remark}
\noindent

Notice that the arguments in Remark~\ref{rk:cond-integr-densite} 
justify that all the integrals in the definition of $M_t$ are well 
defined.
 % fin de rouge

We start with the analysis of Equations \eqref{iter1}-\eqref{iterk}.
\begin{proposition}
\label{iteratesRes}
Same assumptions as  in Theorem \ref{localEX}. Then, for any $k\geq 
1$, Equations \eqref{iter1}-\eqref{iterk} Equations \eqref{iter1}-\eqref{iterk} admit unique weak solutions 
%$(X,\P^k)$ 
up to $T_0$.  For $ k\geq 1$, denote by $\P^k$ the law of $(X_t^k)_{t\leq T_0}$. Moreover, for $t\in (0,T_0]$, the time 
marginals $\P^k_t$ of $\P^k$ have densities $p^k_t$ w.r.t. Lebesgue 
measure on $\R$. Setting $\beta^k=\sup_{t\leq T_0} 
\|b^k(t,\cdot,p^{k-1})\|_{L^\infty(\R)}$ and $b_0:= \|b\|_{L^\infty(\R)}$, one has
$$ \forall 0<t\leq T_0, \quad \|p^k_t\|_{L^\infty(\R)} \leq \frac{C(b_0, T_0)}{\sqrt{t}} \quad \text{and} \quad \beta^k \leq C(b_0, T_0).$$
Finally, there exists a function $p^\infty \in L^\infty([0,T_0];L^1(\R))$ such that 
$$\sup_{t\leq T_0} \|p_t^k-p^\infty_t\|_{L^1(\R)} \to 0, \text{ as } k\to \infty.$$
Moreover,
\begin{equation}
\forall 0<t\leq T_0, \quad \|p^\infty_t\|_{L^\infty(\R)}\leq   \frac{C(b_0, T_0)}{\sqrt{t}}.
\label{ineq:limitEst}
\end{equation}
\end{proposition}
% the sequence $(p^k)_{k\geq 1}:= ((p_t^k)_{t\leq T_0})_{k\geq 1}$ is convergent in $L^\infty((0,T_0);L^1(\R)$}
\begin{proof}
We proceed by induction. 
\paragraph{Case $k=1$.} In view of (H-5), one has $\beta^1\leq b_0 + C_{T_0} $. This implies that the equation \eqref{iter1} has a unique weak solution in $[0,T_0]$ with time marginal densities $(p^1_t(y)dy)_{t\leq T_0}$ which in view of \eqref{densityLp} satisfy
$$\forall t\in(0,T_0], \quad \|p^1_t\|_{L^\infty(\R)}\leq  
\frac{1}{\sqrt{2\pi t}}+\beta^1.$$
\paragraph{Case $k>1$.}
Assume now that the equation for $X^k$ has a unique weak solution and 
assume $\beta^k$  is finite. In addition, suppose that the one 
dimensional time marginals satisfy $$\forall t\in(0,T_0], \quad 
\|p^k_t\|_{L^\infty(\R)}\leq  \frac{1}{\sqrt{2\pi t}}+\beta^k.$$
In view of (H-4), the new drift satisfies
$$|b^{k+1}(t,x;p^k)|\leq b_0 + \int_0^t \|p^k_s\|_{L^\infty(\R)}\|K_{t-s}\|_{L^1(\R)}ds \leq  b_0 + \int_0^t (\frac{1}{\sqrt{2\pi s}}+\beta^k) \|K_{t-s}\|_{L^1(\R)}ds \leq b_0 + C_{T_0} +\beta^k D(T_0).$$
Thus, we conclude that $\beta^{k+1}\leq  b_0 + C_{T_0} +\beta^k D(T_0)$. Therefore, there exists a unique weak solution to the equation for $X^{k+1}$. Furthermore, by \eqref{densityLp}:
$$\forall t\in(0,T_0], \quad \|p^{k+1}_t\|_{L^\infty(\R)}\leq   
\frac{C_{T_0}}{\sqrt{t}}+\beta^{k+1}.$$
%This inductive procedure shows that all the equations \eqref{iter1} and \eqref{iterk} have unique weak solutions. Moreover, their one dimensional marginals have densities with the following estimate:
%$$\forall k \geq 1, \forall t\in(0,T_0]:\quad \|p^k(t,\cdot)\|_\infty\leq  \frac{1}{\sqrt{2\pi t}}+\beta^k.$$
%The relation between $\beta^k$'s is:
Notice that
$$\forall k>1, \quad \beta^{k+1}\leq  b_0 + C_{T_0} +\beta^k D(T_0) \quad \text{ and }\quad \beta^1\leq b_0 + C_{T_0}.$$
%Iterating the previous inequality, we get:
%$$\beta^{k+1}\leq (b_0+C_{T_0}) \sum_{j=0}^{k-1} (D(T_0))^j +(D(T_0))^k \beta^1\leq  (b_0+C_{T_0} ) \sum_{j=0}^{\infty} (D(T_0))^j +b_0 + C_{T_0}.$$
Thus, as by hypothesis $D(T_0)<1$, we have
\begin{equation}
\label{1stTimeHorizonBoundDrift}
\forall k\geq 1, \quad \beta^{k}\leq \frac{b_0 + C_{T_0}}{1-D(T_0)}+b_0 + C_{T_0}
\end{equation}
and
\begin{equation}
\label{1stTimeHorizonBoundDens}
\|p^k_t\|_{L^\infty(\R)} \leq  \frac{C_{T_0}}{\sqrt{t}}+\beta^k\leq  
\frac{C_{T_0}}{\sqrt{t}} + \frac{b_0+C_{T_0}}{1-D(T_0)} +b_0 + C_{T_0}.
\end{equation}
Finally, it remains to prove that the sequence $p^k$ converges in $L^\infty([0,T_0];L^1(\R))$. In order to do so,  we will prove $p^k$ is a Cauchy sequence.

%YFT:As the space is Banach's, the convergence will follow.
Applying the same procedure as in Section \ref{subs:uniq}, one can derive the mild equation for $(p^k_t)_{t\in [0,T_0]}$. Thus, for every $k\geq 1$, the marginals  $(p^k_t)_{t\in (0,T_0]}$ satisfy the mild equation
 \begin{equation}
 \label{mildEQ_k}
 \forall t\in(0,T], \quad p_t^k= g_t\ast p_0-\int_0^t\frac{\partial g_{t-s}}{\partial x} \ast(p^k_s b^k(s,\cdot, p^{k-1}))ds
 \end{equation}
in the sense of the distributions. Assume for a moment that we have proved that for any $0<t\leq T_0$, one has
 \begin{equation}
 \label{p^kconverge}
 \|p_t^k-p_t^{k-1}\|_{L^1(\R)}\leq C_{T_0} \int_0^t \frac{\|p_s^{k-1}-p_s^{k-2}\|_{L^1(\R)}}{\sqrt{s}} \ ds.
 \end{equation}
% $$\|p_t^k-p_t^{k-1}\|_{L^1(\R)}\leq C \int_0^t \frac{\|p_s^{k-1}-p_s^{k-2}\|_{L^1(\R)}}{\sqrt{s}} \ ds.$$
Remember that $\int_0^t f(u_1)\dots \int_0^{u_{k-1}}f(u_k) du_k \ \dots du_1= \frac{1}{k!} \left( \int_0^t f(u) du\right)^k$ for any positive integrable function $f$. Then, iterating \eqref{p^kconverge} one gets,
 $$\|p_t^k-p_t^{k-1}\|_{L^1(\R)} \leq 2 \frac{(C_{T_0} \sqrt{t})^{k-1}}{(k-1)!}.$$
Therefore,  $\sup_{t\leq T_0} \|p_t^k-p_t^{k-1}\|_{L^1(\R)} \to 0 $, as $k \to \infty$ as desired.

%The latter implies the sequence $(p^k)_{k\geq 1}$ is a Cauchy sequence.\\\\
It remains to prove the inequality \eqref{p^kconverge}. In the sequel $C(T_0)>0$ will denote a constant that depends on $T_0$ and may change from line to line. In view of \eqref{mildEQ_k}, one has
\begin{equation}
\label{ineq:difpk}
\begin{split}
\|p_t^k-p_t^{k-1}\|_{L^1(\R)} & \leq \int_0^t \|\frac{\partial g_{t-s}}{\partial x} \ast(p^k_s b^k(s,\cdot, p^{k-1})- p^{k-1}_s b^{k-1}(s,\cdot, p^{k-2})) \|_{L^1(\R)} \ ds\\
& \leq \int_0^t \frac{1}{\sqrt{t-s}} \|b^{k-1}(s,\cdot, p^{k-2})(p_s^k-p_s^{k-1})\|_{L^1(\R)} \ ds\\
& \quad + \int_0^t \frac{1}{\sqrt{t-s}} \|(b^k(s, \cdot, p^{k-1})-b^{k-1}(s,\cdot, p^{k-2}))p_s^k\|_{L^1(\R)} \ ds \\
&=: I + II.
\end{split}
\end{equation}
%\begin{align*}
%& \|p_t^k-p_t^{k-1}\|_{L^1(\R)}\leq \int_0^t \|\frac{\partial g_{t-s}}{\partial x} \ast(p^k_s b^k(s,\cdot, p^{k-1})- p^{k-1}_s b^{k-1}(s,\cdot, p^{k-2})) \|_{L^1(\R)} \ ds\\
%&  \leq \int_0^t \frac{1}{\sqrt{t-s}} \|b^{k-1}(s,\cdot, p^{k-2})(p_s^k-p_s^{k-1})\|_{L^1(\R)} \ ds + \int_0^t \frac{1}{\sqrt{t-s}} \|(b^k(s, \cdot, p^{k-1})-b^{k-1}(s,\cdot, p^{k-2}))p_s^k\|_{L^1(\R)} \ ds\\
%&=: I + II.
%\end{align*}
According to \eqref{1stTimeHorizonBoundDrift}, one has
$$I \leq C(T_0)\int_0^t   \frac{\|p_s^k-p_s^{k-1}\|_{L^1(\R)}}{\sqrt{t-s}} \ ds.$$
 According to \eqref{1stTimeHorizonBoundDens}, one has
 $$II\leq C(T_0) \int_0^t \frac{1}{\sqrt{t-s}\sqrt{s}} \int _0^s \|K_{s-u}\ast(p^{k-1}_u-p^{k-2}_u) \| _{L^1(\R)} \ du  \ ds. $$
 Convolution inequality and Fubini-Tonelli's theorem lead to
 $$II\leq C(T_0) \int_0^t \|p^{k-1}_u-p^{k-2}_u \| _{L^1(\R)}  \int _u^t \frac{1}{\sqrt{t-s}\sqrt{s}} \|K_{s-u}\| _{L^1(\R)}  \ ds  \ du.$$
 %\leq C(T_0) \int_0^t \frac{1}{\sqrt{u}}\|p^{k-1}_u-p^{k-2}_u \| _{L^1(\R)}  \int _u^t \frac{1}{\sqrt{t-s}} \|K_{s-u}\| _{L^1(\R)}  \ ds  \ du. 
 %After using that $\frac{1}{\sqrt{s}}\leq \frac{1}{\sqrt{u}}$, 
 Apply the change of variables $t-s=s'$. It comes,
 $$II \leq C(T_0) \int_0^t \frac{1}{\sqrt{u}}\|p^{k-1}_u-p^{k-2}_u \| _{L^1(\R)}  \int _0^{t-u} \frac{1}{\sqrt{s'}} \|K_{t-u-s'}\| _{L^1(\R)}  \ ds'  \ du. $$ 
 According to (H-4) one has
 $$II\leq C(T_0) \int_0^t \frac{1}{\sqrt{u}}\|p^{k-1}_u-p^{k-2}_u \| _{L^1(\R)}    \ du.$$ 
 Coming back to \eqref{ineq:difpk} and using our above estimates on $I$ and $II$, we obtain
 $$ \|p_t^k-p_t^{k-1}\|_{L^1(\R)}\leq C(T_0)\int_0^t   \frac{\|p_s^k-p_s^{k-1}\|_{L^1(\R)}}{\sqrt{t-s}} \ ds + C(T_0) \int_0^t \frac{1}{\sqrt{u}}\|p^{k-1}_u-p^{k-2}_u \| _{L^1(\R)}    \ du.$$
 We are in the situation 
 $$\Phi(t):= \|p_t^k-p_t^{k-1}\|_{L^1(\R)}\leq A(t) + C \int_0^t \frac{\Phi(s)}{\sqrt{t-s}} \ ds,$$
 where $A(t)\geq 0$ is a bounded increasing function. Iterate this relation and use the monotonicity of $A$. It comes
 $$\Phi(t)\leq C_T A(t) + C^2 \int_0^t \frac{1}{\sqrt{t-s}} \int_0^s \frac{\Phi(u)}{\sqrt{s-u}} \ du \ ds.$$
 Apply Fubini's theorem to get
 $$\Phi(t)\leq C_T A(t) + C^2 \int_0^t \Phi(u)  \int_u^t \frac{1}{ \sqrt{t-s}\sqrt{s-u}} \ ds \ du.$$
Notice that  $\int_u^t \frac{1}{ \sqrt{t-s}\sqrt{s-u}} \ ds = \int_0^1 
\frac{1}{ \sqrt{1-x}\sqrt{x}} \ dx$. Now, apply Gronwall's lemma
to get \eqref{p^kconverge} and the convergence of $p^k$ to 
$p^\infty$.

In order to obtain \eqref{ineq:limitEst}, fix $t\in (0,T]$ and use \eqref{1stTimeHorizonBoundDens} and the fact that the convergence in $L^1(\R)$ implies the almost sure convergence of a subsequence.
\end{proof}
The following is an obvious consequence of the preceding proposition:
\begin{cor}
\label{allSubsSameLim}
Same assumptions as in Proposition \ref{iteratesRes}. Assume that $(\P^k)_{k\geq 1}$ admits a weakly convergent subsequence  $(\P^{n_k})_{k\geq 1}$.
%$(\P^{k_{n_r}})_{r\geq 1}$. 
Denote its limit by $\Q$. Then for any $t\in (0, T_0]$, one has that $\Q_t(dx)= p_t^\infty(x) dx$, where $p^\infty$ is constructed in Proposition \ref{iteratesRes}. 
\end{cor}
%YFT: \begin{proof}
%Let $f\in C_K^\infty(\R)$. Then by weak convergence, 
%$$<f,\Q_t>= \lim_{k\to \infty} <f,p^{n_k}_t>= <f,\tilde{p}_t>+ \lim_{k\to \infty} <f,p^{n_k}_t-\tilde{p}_t >.$$
%In view of Proposition \ref{iteratesRes}, one has
%$$ \lim_{k\to \infty} | <f,p^{n_k}_t-\tilde{p}_t >| \leq  \|f\|_{L^\infty(\R)} \lim_{k\to \infty}\|p^{n_k}_t-\tilde{p}_t \|_{L_1(\R)}=0.$$ 
%Thus, $<f,\Q_t>= <f,\tilde{p}_t>$ which completes the proof.
%\end{proof}
%The following corollary is an obvious consequence of the fact that for $k\geq 1$ and $t \in (0,T_0]$, we have that $p^k_t \in L^1\cap L^\infty(\R)$ and of the just obtained bounds:
%\begin{cor}
%\label{L2iterates}
%For all $k\geq 1$ and $t \in (0,T_0]$ we have that
%$$p^k(t,\cdot) \in L^2(\R) \text{ and } \|p^k(t,\cdot)\|_{L^2(\R)} \leq  \frac{C(b_0, T_0)}{t^{1/4}}.$$
%\end{cor}
\begin{proposition}
\label{solMP}
Same assumptions as  in Theorem \eqref{localEX}. Then, 
\begin{enumerate}[1)]
\item The family of probabilities $(\P^k)_{k> 1} $ is tight.
\item Any weak limit $\P^\infty$ of a convergent subsequence of 
$(\P^k)_{k\geq 1}$ solves $(MP(p_0,T_0,b))$.
\end{enumerate}
\end{proposition}
\begin{proof} 
In view of \eqref{1stTimeHorizonBoundDrift}, we obviously have
$$ \exists C_{T_0}>0,~~~\sup_k \E|X_t^k-X_s^k|^4\leq C_{T_0} |t-s|^2, 
\quad \forall \  0\leq s\leq t\leq T_0.$$
This is a sufficient condition for tightness (see 
e.g. ~\cite[Chap.2, Pb.4.11]{KaratzasShreve}).

%Since all the drifts are uniformly bounded, see \eqref{1stTimeHorizonBoundDrift}, and the diffusion is constant, we can take, for example, $\alpha=4$. Then, by applying  H{\"o}lder's inequality and knowing the higher moments of  Gaussian distribution, we conclude that we will have  $\beta=1$ and some positive $C_{T_0}$ to satisfy the criterion.
%Therefore, there exists a weakly convergent subsequence of 
%$(\P^k)_{k\geq 1} $ that we will still denote by $(\P^k)_{k\geq1} $. 
%Denote its limit by $\P^\infty$.
Let $(\P^{n_k})$ be a weakly convergent subsequence of $(\P^k)_{k\geq 1} $ and let $\P^\infty$ denote its limit. Let us check that $\P^\infty$ solves the martingale problem~$(MP(p_0,T_0,b))$. To simplify the notation, we below write $\P^k$ instead of $\P^{n_k}$ and $\bar{p}^{k-1}$ instead of $p^{n_k-1}$. 
\begin{enumerate}[i)]
\item Each $\P_0^k$ has density $p_0$, and therefore $\P^\infty_0$ also has density $p_0$.
\item Corollary \ref{allSubsSameLim} implies that the time marginals of $\P^\infty$ are absolutely continuous with respect to Lebesgue's measure and satisfy \eqref{defMP:linfnorm}. 
%\item From \eqref{1stTimeHorizonBoundDens} for any $k\geq1$ and $0<t\leq T_0$, we have that 
%\begin{equation}
%\label{iteratesL2}
%p^k_t \in L^2(\R) ~~~ \text{and}~~~ \|p^k_t\|_{L^2(\R)} \leq \frac{C(b_0,T_0)}{t^{1/4}}.
%\end{equation}
%Define the functional $T_t(g)$ by
% $$T_t(g):= \int g(y) \P^\infty_t(dy), \quad g\in C_K(\R).$$
%By weak convergence we have
%$$T_t(g)= \lim_{k\to \infty} \int g(y)p^k_t(y)dy,$$
%%After applying the Cauchy-Schwarz inequality and the bounds on $\|p^k_t\|_2$, that are uniform w.r.t. $k$, we have that:
%and thus
%$$|T_t(g)|\leq\frac{C(b_0,T_0)}{t^{1/4}} \|g\|_{L^2(\R)}.$$
%Therefore, for each $0<t\leq T_0$, $T_t$ is a bounded linear functional 
%on a dense subset of $L^2(\R)$. Thus, $T_t$ can be extended to a linear 
%functional on  $L^2(\R)$. By Riesz-representation theorem, there exists 
%a unique $p^\infty_t \in L^2(\R)$ such that $\|p^\infty_t \|_{L^2(\R)} 
%\leq\frac{C(b_0,T_0)}{t^{1/4}}$ and $p^\infty_t$ is the probability
%density of~$\P^\infty_t(dy)$.
%
%\milica{Notice that $p^\infty_t=\tilde{p}_t $ by Corollary \ref{allSubsSameLim}.}
\item Set
$$M_t:= f(w_t)-f(w_0)-\int_0^t \big[\frac{1}{2} 
\frac{\partial^2f}{\partial x^2}(w_u)+\frac{\partial f}{\partial 
x}(w_u)(b(u,w_u)+ \int_0^u (K_{u-\tau}\ast p^\infty_\tau)(w_u) 
d\tau)\big]du,$$
%is a $\P^\infty$ martingale. That is:
%$$\E_{\P^\infty}(M_t| \F_s) = M_s, \quad 0\leq s \leq t\leq T_0.$$
%This is equivalent to:
We have to prove
$$\E_{\P^\infty}[(M_t-M_s)\phi(w_{t_1}, \dots, w_{t_N})] = 0, \quad \forall \phi \in C_b(\R^N) \text{ and } 0\leq t_1<\dots<t_N<s \leq t\leq T_0, N\geq 1.$$
%Remember that for any $k>1$, probability $\P^k$ is a weak solution to the equation \eqref{iterk} for $X_t^k$. Thus, it solves the martingale problem associated to \eqref{iterk}. Set:
The process
$$M_t^k:= f(w_t)-f(x(0))- \int_0^t \big[\frac{1}{2} 
\frac{\partial^2f}{\partial x^2}(w_u)+\frac{\partial f}{\partial 
x}(w_u)(b(u,w_u)+ \int_0^u  (K_{u-\tau}\ast 
\bar{p}^{k-1}_\tau)(w_u)d\tau\big)]du $$
is a martingale under $\P^k$. Therefore,
it follows that
\begin{align*}
    & 0=\E_{\P^k}[(M^k_t-M^k_s)\phi(w_{t_1}, \dots, w_{t_N})] \\
    & = \E_{\P^k}[\phi(\dots)(f(w_t)-f(w_s))] + 
    \E_{\P^k}[\phi(\dots)\int_s^t \frac{1}{2} 
    \frac{\partial^2f}{\partial x^2}(w_u)du ]\\
    &+\E_{\P^k}[\phi(\dots)\int_s^t \frac{\partial f}{\partial x}(w_u) 
    b(u,w_u) du ]+\E_{\P^k}[\phi(\dots)\int_s^t \frac{\partial 
    f}{\partial x}(w_u) \int_0^u (K_{u-\tau}\ast 
    \bar{p}^{k-1}_\tau)(w_u)\ d\tau \ du ].
\end{align*}
Since $(\P^k)$ weakly converges to $\P^\infty$, the first two terms on the r.h.s. obviously converge.
%\begin{align*}
%    & \lim_{k\to \infty} \E_{\P^k}[\phi(\dots)\int_s^t \frac{\partial 
%f}{\partial x}(w_u) \int_0^u (K_{u-\tau}\ast p^{k-1}_\tau)(w_u)d\tau 
%du 
%] = \\
%    & = \E_{\P^\infty}[\phi(\dots)\int_s^t \frac{\partial f}{\partial 
%x}(w_u) \int_0^u (K_{u-\tau}\ast p^\infty _\tau)(w_u)d\tau du ] .
%\end{align*}
Now, observe that
\begin{align*}
    & \E_{\P^k}[\phi(\dots)\int_s^t \frac{\partial f}{\partial x}(w_u) 
    \int_0^u (K_{u-\tau}\ast\bar{p}^{k-1}_\tau)(w_u)\ d\tau \ du  ] \\
    &- \E_{\P^\infty}[\phi(\dots)\int_s^t \frac{\partial f}{\partial 
    x}(w_u) \int_0^u (K_{u-\tau}\ast p^\infty_\tau)(w_u)\ d\tau \ du ] \\
    & = \big(\E_{\P^k}[\phi(\dots)\int_s^t \frac{\partial f}{\partial 
    x}(w_u) \int_0^u (K_{u-\tau}\ast \bar{p}^{k-1}_\tau)(w_u)\ d\tau \ du  ]\\
    &- \E_{\P^k}[\phi(\dots)\int_s^t \frac{\partial f}{\partial x}(w_u) 
    \int_0^u (K_{u-\tau}\ast p^{\infty}_\tau)(w_u)\ d\tau \ du  ]\big) \\
    & + \big(\E_{\P^k}[\phi(\dots)\int_s^t \frac{\partial f}{\partial 
    x}(w_u) \int_0^u (K_{u-\tau}\ast p^{\infty}_\tau)(w_u)\ d\tau \ du  ] \\
    & -  \E_{\P^\infty}[\phi(\dots)\int_s^t \frac{\partial f}{\partial 
    x}(w_u) \int_0^u (K_{u-\tau}\ast p^\infty_\tau)(w_u)\ d\tau \ du  ]\big) 
    \\
    &=: I  + II.
\end{align*}
%We treat the term $I$. Putting the absolute values and bounding $\phi$ with its maximum, we are left with the expectation with respect to the time marginal in $u$ of $\P^K$. Using the $L^2(\R)$-norm bounds on $p^k_u$:
Now, in view of \eqref{1stTimeHorizonBoundDens} and the definition of $D(T)$ as in \eqref{def:D(T)}, one has
\begin{align*}
& |I|\leq \|\phi\|_{L^\infty(\R)} \int_s^t  \int _0^u \int  
|\frac{\partial f}{\partial x}(x)| |(K_{u-\tau}\ast(\bar{p}^{k-1}_\tau- 
p^\infty_\tau))(x)|p^k_u(x)dx \  d\tau \ du \\
&\leq  \|\phi\|_{L^\infty(\R)} \|\frac{\partial f}{\partial x}\|_{L^\infty(\R)} \int_s^t \frac{C_{T_0}}{\sqrt{u}}\int_0^u \|K_{u-\tau}\|_{L^1(\R)}\|\bar{p}^{k-1}_\tau- 
p^\infty_\tau\|_{L^1(\R)}d\tau \ du\\
&\leq C_{T_0} D(T_0) \|\phi\|_{L^\infty(\R)} \|\frac{\partial f}{\partial x}\|_{L^\infty(\R)}  \sup_{r\leq T_0} \|\bar{p}^{k-1}_r- 
p^\infty_r\|_{L^1(\R)}.
\end{align*}
%$(p^{k-1})_{k\geq 1}$ is a subsequence of a sequence  that converges to $\tilde{p}$ in $L^\infty((0,T_0); L^1(\R))$, one has that
Proposition \ref{iteratesRes} implies that  $I \to 0$ as $k \to \infty$.
Now, to prove that $II\to 0$, it suffices to prove that the functional 
$F: \mathcal{C}([0,T_0];\R)\to \R $ defined by
$$w_. \mapsto \phi(w_{t_1},\dots ,w_{t_N}) \int_s^t  \frac{\partial f}{\partial x}(w_u) \int_0^u \int K_{u-\tau}(w_u-y)p^{\infty}_\tau(y)\ dy\ d\tau \ du $$
is continuous. Let $(w^n)$ a sequence converging in 
$\mathcal{C}([0,T_0];\R)$ to 
$w$. Since $\phi$ is a continuous function, it suffices to show that
\begin{align}
\label{contFunctional1}
    & \lim_{n\to \infty}\int_s^t  \frac{\partial f}{\partial x}(w^n_u) \int_0^u \int K_{u-\tau}(w^n_u-y)p^{\infty}_\tau(y)\ dy\ d\tau \ du\\ \nonumber
    & = \int_s^t  \frac{\partial f}{\partial x}(w_u) \int_0^u \int K_{u-\tau}(w_u-y)p^{\infty}_\tau(y)\ dy\ d\tau \ du .
\end{align}
%Again, in both integrals in \eqref{contFunctional1} we will consider only the case when $\tau < u$.\\\\
For $(u,\tau) \in [s,t] \times [0,t]$, set
$$h_{u,\tau}(x):= \mathbbm{1}\{\tau< u\}\frac{\partial f}{\partial x}(x_u) \int K_{u-\tau}(x-y)p^{\infty}_\tau(y)dy. $$ 
The hypothesis (H-2) implies the continuity of $h_{u,\tau}$ on $\R$.
%$$f^n_3(u,\tau):= \mathbbm{1}\{\tau< u\}\frac{\partial f}{\partial x}(w^n_u) \int K_{u-\tau}(w^n_u-y)p^{\infty}_\tau(y)dy. $$ 
%As by $(H)$, for $u-\tau >0$, $K_{u-\tau}$ is a  bounded function on $\R$ by a finite constant that may depend on $u-\tau$,
%$$|\mathbbm{1}\{\tau< u\} K_{u-\tau}(x_n(u)-y)p^{\infty}_\tau(y)|\leq \mathbbm{1}\{\tau< u\}C(u-\tau)p^{\infty}_\tau(y) \in L^1(\R).$$
%%Since $w^n$ converges pointwise to $w$ and $K$ satisfies (H-2), by 
%%Lebesgue's dominated convergence theorem,
%%$$\lim_{n\to \infty } f^n_3(u,\tau) = \mathbbm{1}\{\tau< u\}\frac{\partial f}{\partial x}(w_u) \int K_{u-\tau}(w_u-y)p^{\infty}_\tau(y)dy.$$
Furthermore, 
%using Cauchy-Schwarz inequality and the bounds on of $\frac{\partial f}{\partial x}$ and $\|p^\infty_\tau\|_2 $:
$$|h_{u,\tau}(x) |\leq C \mathbbm{1}\{\tau< u\}\|p^\infty_\tau\|_{L^\infty(\R)} 
\|K_{u-\tau}\|_{L^1(\R)}\leq \frac{C}{{\sqrt{\tau}}} \mathbbm{1}\{\tau< u\} 
\|K_{u-\tau}\|_{L^1(\R)}.$$
In view of (H-4), we apply the theorem of dominated convergence to conclude \eqref{contFunctional1}.
%%In view of (H-4), 
%%$$ \int_s^t \int_0^t \frac{1}{\tau^{1/4}}\mathbbm{1}\{\tau< 
%%u\}\|K_{u-\tau}\|_{L^2(\R)} d\tau~du < \infty. $$
%%Thus, by Lebesgue's dominated convergence theorem, \eqref{contFunctional1} follows.
%\begin{align*}
%    & \lim_{n\to \infty} \int_s^t \int_0^t \mathbbm{1}\{\tau< u\}\frac{\partial f}{\partial x}(w^n_u) \int K_{u-\tau}(w^n_u-y) p^{\infty}_\tau(y)dy d\tau du\\
%    & = \int_s^t \int_0^t \mathbbm{1}\{\tau< u\}\frac{\partial f}{\partial x}(w_u) \int K_{u-\tau}(w_u-y) p^{\infty}_\tau(y)dy d\tau du,
%\end{align*}
This ends the proof.
%This proves \eqref{contFunctional1}. We conclude that $II\to 0$ as $k\to \infty$. Thus, $\P^\infty$ satisfies the point $(iii)$ of the definition of the $(MP(p_0,T_0,b))$.
\end{enumerate}
\end{proof}
\paragraph{Proof of Theorem \ref{localEX}\\}
Proposition \ref{solMP} implies the existence of a weak solution 
$(\Omega, \F, \P, (\F_t), X, W )$   to \eqref{NLSDEsimple} up to time 
$T_0$. Thus, the marginals  $\P \circ X_t^{-1}=:p_t$ satisfy 
$\|p_t\|_{L^\infty(\R)}\leq \frac{C}{\sqrt{t}}$, $t \in (0,T_0]$.
In addition, as $|B(t,x;p)|\leq C(T_0)$, one has that $(\Omega, \F, \P, (\F_t), X, W )$ is the unique weak solution of the linear SDE
\begin{equation}
\label{LinearSDE}
d\tilde{X}_t= b(t,\tilde{X}_t)dt+ B(t,\tilde{X}_t; p)dt + dW_t, \quad t\leq T_0.
\end{equation}

Now suppose that there exists another weak solution $(\hat{\Omega}, \hat{\F}, \hat{\P}, (\hat{\F}_t), \hat{X}, \hat{W} )$ to \eqref{NLSDEsimple} up to $T_0$ and denote $\hat{\P}\circ \hat{X}_t^{-1}(dx)=\hat{p}_t(x)dx$. 
%Therefore, it is true that the one dimensional marginals of the constructed solution $\P^\infty$ are densities satisfying \eqref{density_space}.\\\\
By Proposition \ref{MeqUniq} we have $\hat{p}_t= p_t$, for $t\leq T_0$. Therefore,  $(\hat{\Omega}, \hat{\F}, \hat{\P}, (\hat{\F}_t), \hat{X}, \hat{W} )$ is a weak solution to \eqref{LinearSDE}, from which $\hat{\P}\circ \hat{X}^{-1}= \P \circ X^{-1}$.
% Let $\tilde{b}(t,x):=b(t,x)+\int _0^t (K_{t-s}\ast p^\infty_s)(x)ds$ and consider the SDE
%\begin{equation}
%\label{LinearSDE}
%d\tilde{X}_t= \tilde{b}(t,\tilde{X}_t)dt + dW_t, \quad t\leq T_0.
%\end{equation}
%Any weak solution has probability distribution $\tilde{\P}$. Now, assume that there exists another weak solution $\P$ to \eqref{NLSDEsimple} up to $T_0$, such that $\P_t(dy)= p_t(y)dy$ and $p_t$ satisfies \eqref{density_space} for any $t\leq T_0$. Then, by Proposition \ref{MeqUniq}, $p_t(y)dy=p^\infty_t(y)dy$. Therefore, $\P$ is another weak solution to \eqref{LinearSDE}. We conclude that $\P=\P^\infty$. The Theorem \ref{localEX} is proved.

%%%%%%%%%%%%%%%%%%%%%%%%%%%%%%%%%%%%%%%%%%%%%%%%%%%%%%
\section{Proofs of Theorem~\ref{mainTh} and Corollary \ref{cor:forPS}: A global existence and 
uniqueness result for Equation \eqref{NLSDEsimple}}

We now construct a solution for an arbitrary time horizon $T>0$. We 
will do it by restarting the equation after the already fixed $T_0$. We 
start with $T=2 T_0$. Then, we will see how to generalize this 
procedure for an arbitrary $T>0$. 

Throughout this section, we denote by $\Omega_0$ the canonical space 
$\mathcal{C}([0,T_0];\R)$ and by $\B_0$ its Borel $\sigma$- field.
We denote by $\Q^1$ the probability distribution of 
the unique weak solution to~\eqref{NLSDEsimple} up to $T_0$ constructed 
in the previous section.

\subsection{Solution on $[0,2T_0]$}
\begin{proposition}
\label{localEx2T0}
Let $T_0>0$ be such that $D(T_0)<1$. Assume $p_0 \in L^1(\R)$ and let 
$b\in L^\infty([0,2T_0]\times \R)$ be continuous w.r.t. the space 
variable. Under the hypothesis $(\text{H})$, Equation 
\eqref{NLSDEsimple} admits a unique weak solution up to $2T_0$.
% which is unique 
%in the sense of Theorem~\ref{mainTh}.
\end{proposition}
\noindent We start with analyzing the dynamics of \eqref{NLSDEsimple} 
after $T_0$ and informally explaining the construction of a solution 
between $T_0$ and $2T_0$. Assume, for a while, that Proposition 
\ref{localEx2T0} holds true. Denote the density of $X_t$ by $p^1_t$, 
for $t\leq T_0$ and by $p^2_t$, for $t\in (T_0,2T_0]$. 
%Then,
%$$ \|p^1(t,\cdot)\|_\infty \leq \frac{C(T_0, b_0)}{\sqrt{t}} .$$
 Let $t \geq 0$. In view of Equation \eqref{NLSDEsimple}, we would have
 $$X_{T_0+t}= X_{T_0}+\int_{T_0}^{T_0+t} b(s,X_s)ds + 
 \int_{T_0}^{T_0+t} \int_0^s (K_{s-\theta} \ast p_\theta)(X_s)d\theta 
 ds + W_{T_0+t}-W_{T_0}.$$
 Observe that
%\noindent Let us transform the integral coming from the non-linear drift part. We will split the domain of integration in the inner integral in time in two parts. In the part until $T_0$ we will plug $p^1_t$, while in the second part we are left with $p^2_t$:
 \begin{align*}
     & \int_{T_0}^{T_0+t} \int_0^s (K_{s-\theta} \ast 
     p_\theta)(X_s)d\theta ds =  \int_{T_0}^{T_0+t} \int_0^{T_0} 
     (K_{s-\theta} \ast p^1_\theta)(X_s)d\theta ds \ + 
     \int_{T_0}^{T_0+t} \int_{T_0}^s (K_{s-\theta} \ast 
     p^2_\theta)(X_s)dsdt\\
     &=: B_1 + B_2.
 \end{align*}
% In $B_1$, we apply the change of variables $s-T_0=s'$:
In addition,
 $$B_1= \int_0^{t}  \int_0^{T_0} (K_{T_0+s'-\theta} \ast 
 p^1_\theta)(X_{T_0+s'})d\theta ds', $$
%In $B_2$, we firstly apply the same change of variables:
and
$$B_2=\int_{0}^{t} \int_{T_0}^{T_0+s'} (K_{T_0+s'-\theta} \ast 
p^2_\theta)(X_{T_0+s'})d\theta ds'=\int_{0}^{t} \int_{0}^{s'} 
(K_{s'-\theta'} \ast p^2_{T_0+\theta'})(X_{T_0+s'}) d\theta' ds'.$$
%and then $\theta-T_0=\theta'$
Now set $Y_t:= X_{T_0 +t}$ and $\tilde{p}_t(y):= p^2_{T_0 + t}(y)$. 
Consider the new Brownian motion $\overline{W}_t:= W_{T_0+t}-W_{T_0}$. 
It comes:
%In addition, $B_t$  is still a Brownian motion with respect to the original probability space. Then the expression for $X_{T_0 + t}$ becomes:
\begin{equation*}
    Y_t =Y_0 + \int_0^tb(s+T_0,Y_s)ds+ \int_0^{t}  \int_0^{T_0} 
    (K_{T_0+s'-\theta} \ast p^1_\theta)(Y_{s})d\theta ds + 
    \int_{0}^{t} \int_{0}^s (K_{s'-\theta'} \ast 
    \tilde{p}_\theta)(Y_{s})d\theta ds + \overline{W}_t,
\end{equation*}
for $t\in[0,T_0]$. Setting
%We can see the preceding expression as a new stochastic process that has the exact same law as the process $X$ after the time $T_0$. The drift of the new equation has three components. The first one will depend on the function $b$, second of the measure $p^1$ and the third of the law of the process. Define:
$$b_1(t,x,T_0):= \int_0^{T_0} (K_{T_0+t-s}\ast p^1_s)(x)ds \quad \text{ 
and } \quad \tilde{b}(t,x):=b(T_0+t,x),$$
we have
\begin{equation}
\label{RestartedEquation}
\begin{cases}
& d Y_t=\tilde{b}(t,Y_t)dt+ b_1(t,Y_t,T_0)dt + \Big\{ \int_{0}^t 
(K_{t-s}\ast \tilde{p}_s)(Y_t) ds \Big \}dt + d\overline{W}_t, \ t\leq 
T_0, \\
& Y_0\sim p^1_{T_0}(y) dy, \  Y_s \sim \tilde{p}_s(y)dy .
\end{cases}
\end{equation}
%The above discussion implies that a weak solution to 
%\eqref{RestartedEquation} up to $T_0$ would be the conditional  law of 
%the solution to \eqref{NLSDEsimple} up to $2T_0$ knowing $\F_{T_0}$, 
%if 
%that solution existed.\milica{(Question: Does this have sense?)}\\\\

To prove Proposition \ref{localEx2T0} we construct a weak 
solution to \eqref{RestartedEquation} on $[0,T_0]$ and suitably paste 
its probability distribution with $\Q^1$. We then prove that 
the so defined measure solves the desired non-linear martingale 
problem. Notice that the SDE~\eqref{RestartedEquation} is of 
the same type as \eqref{NLSDEsimple}.

\begin{lemma}
Same assumptions as in Proposition \ref{localEx2T0}. Denote by $p^1_t$ 
the time marginals of $\Q^1$. Set $b_1(t,x,T_0):= \int_0^{T_0} 
(K_{T_0+t-s}\ast p^1_s)(x)ds$ and $\tilde{b}(t,x):=b(T_0+t,x)$. Then, 
Equation \eqref{RestartedEquation} admits a unique weak solution up to 
$T_0$.
\end{lemma}

\begin{proof}
Let us check that we may apply Theorem \ref{localEX} to 
\eqref{RestartedEquation}.

Firstly, by construction the initial law $p^1_{T_0}$ of $Y$ 
satisfies  the assumption of Theorem \ref{localEX}. 
Secondly, the role of the additional drift $b$ is now played by the sum 
of 
the two linear drifts, $\tilde{b}$ and $b_1$. By hypothesis, 
$\tilde{b}$ is 
bounded in $[0,T_0]\times \R$ and continuous in the space variable. 
Using \eqref{density_infty} and (H-6) we conclude that $b_1$ is bounded 
uniformly in $t$ and $x$ since
$$ |b_1(t,x,T_0)|\leq C_{T_0} 
\int_0^{T_0}\frac{\|K_{T_0+t-s}\|_{L^1(\R)}}{\sqrt{s}} ds
< C_{T_0}. $$
To show that $b_1(t,x,T_0)$ is continuous w.r.t.~$x$ we use~(H-2)
and proceed as at the end of the proof of Proposition~\ref{solMP}. 

We now are in a position to apply Theorem \ref{localEX}: There exists a 
unique weak solution 
 to \eqref{RestartedEquation} up to $T_0$.
\end{proof}
%\begin{cor}
%The one dimensional time marginals have the similar bounds as $p^1$, where $b_0$ is replaced by its analog. 
%\end{cor}
Denote by $\Q^2$ the probability distribution of the 
process~$(Y_t,t\leq T_0)$. Notice that $\Q^2$ is the solution to the 
martingale problem $(MP(p^1_{T_0},T_0,\tilde{b} + b_1))$.
%We give here a definition of a martingale problem related to \eqref{RestartedEquation}. It is just a modification of  Definition \ref{defMP}. It will be useful for us later on. 
%\begin{mydef}
%\label{defMP'}
%A probability measure $\Q$ on the canonical space 
%$\mathcal{C}([0,T_0];\R)$ is a solution to the non-linear martingale 
%problem $(MP')$ if:
%\begin{enumerate}[(i)]
%\item $\Q_0 = p^1_{T_0}$.
%\item For any $t\in (0,T_0]$, the one dimensional time marginals of $\Q$, denoted by $\Q_t$, have densities $q_t$ w.r.t. Lebesgue measure on $\R$. In addition, they satisfy $q_t\in L_2(\R)$ and 
%$$\sup_{t\in (0,T_0]}\|q_t\|_2 \leq C.$$
%\item For all $f \in C_b^2(\R)$, the process $(M_t)_{t\leq T_0}$, defined by:
%\begin{align*}
%    & M_t:=f(w_t)-f(w_0)-\int_0^t \big[\frac{1}{2} 
%\frac{\partial^2f}{(\partial x)^2}(w_u)+\\
%    & +\frac{\partial f}{\partial x}(w_u)( \int_0^u \int 
%K_{u-\tau}(w_u-y)q_\tau(y)dyd\tau +b_1(u,w_u,T_0) + 
%\tilde{b}(u,w_u))\big]du\\
%\end{align*}
%is a $\Q$-martingale.
%\end{enumerate}
%\end{mydef}
\paragraph{A new measure on $\mathcal{C}([0,2T_0];\R)$.}
Let $\Q^1$, $\Q^2$ and $(p^1_t)$ be as above. Let $(p^2_t)$
denote the time marginal densities of $\Q^2$.
%By means of this two measures, we would like to define a new measure $\Q$ on the space $\Omega:= C[0,2T_0]$. The space $\Omega$ is a Polish space with respect to uniform convergence metric. The Borel $\sigma$- field on $\Omega$ is denoted by $\B$ and $(\B_t)_{t\leq 2T_0}$ is its natural filtration. \\\\
In particular, $\Q^2_0=\Q^1_{T_0}$, i.e. $p^2_0(z)dz=p^1_{T_0}(z)dz$. 
Define the mapping $X^0$ from $\Omega_0$ 
to $\R$ as $X^0(w) := w_0$. 
Using~\cite[Thm.3.19, Chap.5]{KaratzasShreve} to justify the 
introduction of regular conditional probabilities, for each $y\in \R$ 
we define the probability measure 
$\Q^2_y$ on $(\Omega_0, \B_0)$ by 
$$ \forall A\in \B_0,~~\Q^2_y(A)= \P^2(A|X^0=y).$$
In particular,
$$\Q^2_y(w\in \Omega_0, w_0=y)=1.$$

We now set $\Omega:=\mathcal{C}([0,2T_0];\R)$.
For $w^1_\cdot, w^2_\cdot \in \Omega_0$ we define the concatenation $w= 
w^1\otimes_{T_0} w^2 \in \Omega$ of these two paths as the function
in $\Omega$ defined by
\begin{equation*}
    \begin{cases}
    & w_\theta = w^1_\theta,  \quad \quad \ 0\leq \theta\leq T_0,\\
    & w_{\theta +T_0} = w^1_{T_0} + w^2_{\theta} - w^2_{0}, \quad 0\leq \theta\leq t-T_0.\\
    \end{cases}
\end{equation*}
On the other hand, for a given path $w\in \Omega$, the two paths $w^1_\cdot, w^2_\cdot \in \Omega_0$ such that $w= w^1\otimes_{T_0} w^2$ are
\begin{equation*}
    \begin{cases}
& w^1_\theta= w_\theta, \quad \quad \ 0\leq \theta\leq T_0, \\
& w^2_{\theta}= w_{T_0+\theta}, \quad 0\leq \theta\leq T_0.\\
\end{cases}
\end{equation*}
We define the probability distribution $\Q$ on 
$\Omega$ equipped with its Borel $\sigma$--field as 
follows. For any continuous and bounded functional $\varphi$ on 
$\Omega$,
\begin{equation}
\label{QpartP2}
\E_{\Q}[\varphi]= \int_\Omega \varphi(w)\Q(dw):= \int 
_{\Omega_0}\int_\R \int _{\Omega_0} \varphi (w^1\otimes_{T_0} w^2 
)\Q^2_y(dw^2) p^1_{T_0}(y)dy \Q^1(dw^1).
\end{equation}
%The idea behind this definition is that we simply fix one path up to 
%$T_0$ and integrate the rest with respect to $P^2_y$. Then we take 
%into 
%account the initial condition of $\Q^2$. What we get is a functional 
%of 
%the fixed part $w^1$. Then we integrate that functional with respect 
%to 
%the measure $\Q^1$.\\\\
Notice that if $\varphi$ acts only on the part of the path up to $t\leq 
T_0$ of any $w_{\cdot} \in \Omega$, then
\begin{equation}
\label{QpartP1}
\E_{\Q}[\varphi((w_\theta)_{\theta\leq t})]= \int _{\Omega_0} 
\varphi((w_\theta)_{\theta\leq t}) \Q^1(dx)= 
\E_{\Q^1}[\varphi((w_\theta)_{\theta\leq t})].
\end{equation}
%In addition, in view of \eqref{QpartP1} for any $T_0\leq t\leq 2 T_0$ 
%it holds that
%%$$\E_\Q[\varphi((w_\theta)_{\theta\leq 
%%t})]=\E_\Q[\E_\Q[\varphi((w_\theta)_{\theta\leq t})| \B_{T_0}]]. $$
%%The conditional expectation can be seen as a $ \B_{T_0}$ measurable 
%%random variable $\Phi$. 
%$$\E_\Q[\E_\Q[\varphi((w_\theta)_{\theta\leq t})| \B_{T_0}]]= \E^\Q[\Phi((w^1_\theta)_{\theta\leq T_0})]= \int_{\Omega_0} \Phi(w^1)d\P^1(dw^1).$$
%On the other hand,  by \eqref{QpartP2},
%$$\E^\Q[\varphi((w_\theta)_{\theta\leq t})]=\int _{\Omega_0}\int_\R \int _{\Omega_0} \varphi (w^1\otimes_{T_0} w^2 )\P^2_y(dw^2) p^2_{T_0}(y)dy \P^1(dw^1).$$
%By e.g. \cite[Section 5, Thm $5.3.18$]{KaratzasShreve} there exists a unique regular conditional probability for $\Q$ given $\B_{T_0}$. Therefore, for any $w^1 \in \Omega_0$:
%\begin{equation}
%\label{condProb}
%\E^\Q[\varphi((w_\theta)_{\theta\leq t})| \B_{T_0}] = \int_\R \int _{\Omega_0} \varphi (w^1\otimes_{T_0} w^2 )\P^2_y(dw^2) p^2_{T_0}(y)dy.
%\end{equation}
%%%%%%%%%%%%%%%%%%%%%%%%%%%%%%%%%%%%%%%%%%%%%%%%%%%%%%%%
%\paragraph{The non linear martingale problem in 
%$\mathcal{C}([0,2T_0];\R)$}
%In this section we would like to prove that the new measure $\Q$ solves the non linear martingale problem related to \eqref{NLSDEsimple} when $t\leq 2T_0$, that is $(MP(p_0,2T_0,b))$.
\paragraph{Proof of Proposition \ref{localEx2T0}.}
Let us prove that the probability measure $\Q$ solves the non--linear 
martingale problem $(MP(p_0,2T_0,b))$ on the canonical space 
$\mathcal{C}([0,2T_0];\R)$. 
%  We repeat its definition:
%\begin{mydef}
%A probability measure $\Q$ on the canonical space 
%$\mathcal{C}([0,2T_0];\R)$ is a solution to the non-linear martingale 
%problem (MP) if:
%\begin{enumerate}[(i)]
%\item $\Q_0 = p_0$.
%\item For all $t\in (0,2T_0]$,  the one dimensional time marginals of $\Q$, denoted by $\Q_t$, have densities $q_t$ w.r.t. Lebesgue measure on $\R$. In addition, they satisfy $q_t\in L_2(\R)$ and 
%$$\sup_{t\in (0,2T_0]}\|q_t\|_2 \leq C.$$.
%\item For all $f \in C_b^2(\R)$, the process $(M_t)_{t\leq 2T_0}$, defined by:
%$$M_t:=f(w_t)-f(w_0)-\int_0^t \big[\frac{1}{2} 
%\frac{\partial^2f}{(\partial x)^2}(w_u)+\frac{\partial f}{\partial 
%x}(w_u)(b(u,w_u) +\int_0^u \int 
%K_{u-\tau}(w_u-y)q_\tau(y)dyd\tau)\big]du $$
%is a $\Q$-martingale.
%\end{enumerate}
%\end{mydef}
%\begin{proposition}
%\label{NLMP2T0}
%The probability measure $\Q$ solves the non linear martingale problem 
%$(MP(p_0,2T_0,b))$ on the canonical space $\mathcal{C}([0,2T_0];\R)$. 
%\end{proposition}
\begin{enumerate}[i)]
\item By \eqref{QpartP1}, it is obvious that $\Q_0= \Q^1_0$. By 
construction,  $\Q_0^1$ has density $p_0$.
\item Next, let us characterize the one dimensional time marginals of 
$\Q$. For $f\in C_b(\R)$ and $t\in [0,2T_0]$, consider the functional 
$\varphi(w)= f(w_t)$, for any $x \in \mathcal{C}([0,2T_0];\R)$. 
%Then, it is always true that:
%$$\E_\Q[\phi(w)]= \int_\R f(z)\Q_t(dz).$$
For $t\leq T_0$, by \eqref{QpartP1},
    $$\E_\Q[\varphi(w)]= \int_{\Omega_0} f(w_t) \Q^1(dx)= \int_\R 
    f(z)p_t^1(z)dz.$$
Therefore, $\Q_t(dz)= p_t^1(z)dz.$

For $T_0\leq t\leq 2 T_0$, by \eqref{QpartP2},
    $$\E_\Q[\varphi(w)]= \int_{\Omega_0} \int_\R \int_{\Omega_0} 
    f(w^2_{t-T_0}) \Q^2_y(dw^2)p^2_{T_0}(y)dy \Q^1(dw^1)= \int_\R 
    \int_\R f(z)\Q^2_{y,t-T_0}(dz) p^1_{T_0}(y)dy.$$
By Fubini's theorem:
    $$\E_\Q[\varphi(w)]=\int_\R f(z) \int_\R \Q^2_{y,t-T_0}(dz) 
    p^1_{T_0}(y)dy.$$
Since  $\Q^2_0= p^1_{T_0}$ we deduce
    $$\E_\Q[\varphi(w)]=\int_\R f(z) p^2_{t-T_0}(z)dz,$$
which shows that $\Q_t(dz)= p_{t-T_0}^2(z)dz.$
   
%    \milica{I am not sure that the marginals of $\P_y$ have densities, that is why I wrote them like measures. Then once we integrate w.r.t. $y$ they turn into the marginals of $\P^2$ that now have densities.}\\
Therefore, the one dimensional marginals of $\Q$ have densities $q_t$ which, by construction, satisfy $\|q_t\|_{L^\infty(\R)}\leq \frac{C}{\sqrt{t}}$.
\item It remains to show that $(M_t)_{t\leq 2T_0}$ defined as 
$$M_t:=f(w_t)-f(w_0)-\int_0^t \big[\frac{1}{2} 
\frac{\partial^2f}{\partial x^2}(w_u)+\frac{\partial f}{\partial 
x}(w_u)(b(u,w_u)+ \int_0^u \int K_{u- \tau}(w_u-y) 
q_\tau(y)dyd\tau\big)]du $$
is a $\Q$-martingale, i.e. $\E_\Q(M_t|\B_s)=M_s$.
\begin{enumerate}
\item Let $s \leq t\leq T_0:$\\
For any $n\in N$, any continuous bounded functional $\phi$ on $\R^n$, and any ${t_1}\leq \dots \leq {t_n}<s\leq t\leq T_0$, by \eqref{QpartP1}:
$$\E_\Q(\phi(w_{t_1}, \dots, w_{t_n}) (M_t-M_s))= 
\E_{\Q^1}(\phi(w_{t_1}, \dots, w_{t_n}) (M_t-M_s)).$$
As $\Q^1$ solves the $(MP(p_0,T_0,b))$ up to $T_0$,
$$\E_\Q(\phi(w_{t_1}, \dots, w_{t_n}) (M_t-M_s))=0.$$
\item For $s\leq T_0\leq t\leq2T_0$,\\
$$\E_\Q(M_t|\B_s)=\E_\Q[\E_\Q(M_t|\B_{T_0})|\B_s].$$
Let us prove that $\E_\Q(M_t|\B_{T_0})=M_{T_0}$. Notice that
%We will  again check if  $\E^\Q(\phi(w_{t_1}, \dots 
%w_{t_n})(M_t-M_{T_0}))=0$, for any continuous bounded functional 
%$\phi$, any $n\in N$ and any ${t_1}\leq \dots \leq {t_n}\leq  T_0\leq 
%t.$ We look at the difference $M_t-M_{T_0}$:
\begin{align*}
   & M_t-M_{T_0} = f(w_t)- f(w_{T_0})- \int_{T_0}^t 
   \frac{1}{2}\frac{\partial^2f}{\partial x^2}(w_u)du - 
   \int_{T_0}^t\frac{\partial f}{\partial x}(w_u)b(u,w_u)du \\
   & -\int_{T_0}^t\frac{\partial f}{\partial x}(w_u) \int_0^u \int 
   K_{u-\tau}(w_u-y)q_\tau(y)dyd\tau du.\\
\end{align*}
Write the last integral as
\begin{align*}
    & \int_{T_0}^t\frac{\partial f}{\partial x}(w_u) \int_0^u \int 
    K_{u-\tau}(w_u-y)q_\tau(y)dyd\tau du= \int_{T_0}^t\frac{\partial 
    f}{\partial x}(w_u) \int_0^{T_0}\int 
    K_{u-\tau}(w_u-y)p^1_\tau(y)dyd\tau du\\
    & + \int_{T_0}^t\frac{\partial f}{\partial x}(w_u) 
    \int_{T_0}^u\int K_{u-\tau}(w_u-y)p^2_{\tau-T_0}(y)dyd\tau du\\
    & =: I_1+I_2.\\
\end{align*}
Now,
%In $I_1$, employ the change of variables $u-T_0=u'$ not renaming the variable:
    $$I_1= \int_{0}^{t-T_0}\frac{\partial f}{\partial x}(w_{u+T_0}) 
    \int_0^{T_0}\int K_{u+T_0- \tau}(w_{u+T_0}-y) p^1_\tau(y)dyd\tau 
    du.$$
    For $w\in \Omega$ identify $w^1,w^2 \in \Omega_0$ such that 
    $w=w^1\otimes_{T_0} w^2$. Then,
    $$I_1= \int_{0}^{t-T_0}\frac{\partial f}{\partial x}(w^2_u) 
    \int_0^{T_0} (K_{u+T_0-\tau} \ast p_\tau^1)(w^2_u)d\tau du= 
    \int_{0}^{t-T_0}\frac{\partial f}{\partial x}(w^2_u) b_1(u,w^2_u, 
    T_0) du.$$
   Proceeding as above,
   \begin{align*}
   &I_2= \int_{0}^{t-T_0}\frac{\partial f}{\partial x}(w_{u+T_0}) 
   \int_0^u  \int K_{u- \tau}(w_{u+T_0}-y) p^2_\tau(y)dyd\tau du\\
   & = \int_{0}^{t-T_0}\frac{\partial f}{\partial x}(w^2_u) \int_0^u 
   \int  K_{u- \tau}(w^2_u-y)p^2_\tau(y)dyd\tau du.
   \end{align*}
   Similarly
   \begin{align*}
   & \int_{T_0}^t\frac{\partial f}{\partial 
   x}(w_u)b(u,w_u)du=\int_0^{t-T_0}\frac{\partial f}{\partial 
   x}(w_{u+T_0})b(u+T_0,w_{u+T_0})du\\
   & = \int_0^{t-T_0}\frac{\partial f}{\partial 
   x}(w^2_u)b(u+T_0,w^2_u)du=\int_0^{t-T_0}\frac{\partial f}{\partial 
   x}(w^2_u)\tilde{b}(u,w^2_u)du . 
   \end{align*}

   % There we identify $\tilde{b}$,
%     In $I_2$, employ the same change of variables $u-T_0=u'$ and then also $\tau-T_0=\tau'$ not renaming the variables:
   
% Notice that $w_t=w^2_{t-T_0}$ and $w(T_0)=x^2(0)$. Perform the 
%change 
%of variables $u-T_0=u'$ in the integral involving the second 
%derivative 
%of $f$. Not renaming the variables, we have:
It comes:
    \begin{align*}
        & M_t-M_{T_0}= f(w^2_{t-T_0})- f(w^2_0)- \int_0^{t-T_0} 
        \frac{1}{2}\frac{\partial^2f}{\partial 
        x^2}(w^2_u)du-\int_0^{t-T_0}\frac{\partial f}{\partial 
        x}(w^2_u)\tilde{b}(u,w^2_u)du\\
        & - \int_{0}^{t-T_0}\frac{\partial f}{\partial 
        x}(w^2_u)\big[b_1(u,w^2_u, T_0)+\int_0^u  \int K_{u- 
        \tau}(w^2_u-y)p^2_\tau(y)dyd\tau\big]du=: F(w^2).\\
    \end{align*}
%    Thus, $ M_t-M_{T_0}=: \varphi (w^2)$ acts only on the second part 
%of the path and  $\phi$ acts only on the first part. 
    By definition of the measure $\Q$, 
    \begin{equation*}
   \E^\Q(\phi(w_{t_1}, \dots w_{t_n})(M_t-M_{T_0}))= 
   \int_{\Omega_0} 
   \phi(w^1_{t_1}, \dots, w^1_{t_n})\int_\R \int_{\Omega_0}  F (w^2) 
   \Q^2_y(dw^2) p^1_{T_0}(y)dy \Q^1(dw^1).
    \end{equation*}
    By the definition of $\Q^2$:
    $$\E^\Q(\phi(w_{t_1}, \dots w_{t_n})(M_t-M_{T_0}))= 
    \int_{\Omega_0} 
    \phi(w^1_{t_1}, \dots, w^1_{t_n})\int_{\Omega_0}  F (w^2) 
    \Q^2(dw^2) \Q^1(dw^1).$$
   As $\Q^2$ solves $(MP(p^1_{T_0},T_0,\tilde{b} + b_1))$, one has
    % we see from the form of $\varphi$, that we actually have the martingale from the definition $(MP(p^1_{T_0},T_0,\tilde{b} + b_1))$. 
    $$\E_{\Q^2}(\varphi)= \int_{\Omega_0}F (w^2)\Q^2(dw^2)= 0.$$
%    By the definition of $\P^2_y$:
%    $$0= \E_{\P^2}(\varphi)= \int_\R   \E_{\P_y^2}(\varphi)p^2_0(y) 
%dy= \int_{\R} \int_{\Omega_0}\varphi 
%(w_2)\P_y^2(dw^2)p^1_{T_0}(y)dy.  $$
%     Coming back to \eqref{eq1nlmp2T0}:
%   $$ \E^\Q(\phi(w_{t_1}, \dots w_{t_n})(M_t-M_{T_0}))=0.$$
    Finally, we conclude that $\E_\Q(M_t|B_{T_0})=M_{T_0}$ and 
    therefore $\E_\Q(M_t|\B_s)=M_s$ for all $s\leq T_0 \leq t \leq 
    2T_0$.
    \item For $T_0\leq s\leq t\leq 2 T_0:$ we may rewrite the difference $M_t-M_s$ in the same manner:
    \begin{align*}
        & M_t-M_{s}= f(w^2_{t-T_0})- f(w^2_{s-T_0})- 
        \int_{s-T_0}^{t-T_0} 
        \frac{1}{2}\frac{\partial^2f}{\partial x^2}(w^2_u)du\\
        & - \int_{s-T_0}^{t-T_0}\frac{\partial f}{\partial 
        x}(w^2_u)[b(u,w^2_u))+b_1(u,w^2_u, T_0)+\int_0^u  \int 
        K_{u-\tau}(w_{u+T_0}-y)p^2_\tau(y)dyd\tau]du\\
        &=:F (w^2).
       \end{align*}
       %Again, we notice that $\varphi (w^2) $ is the same difference 
       %as in the martingale problem $(MP(p^1_{T_0},T_0,\tilde{b} + 
       %b_1))$. 
       Now, take $t_1 \leq \dots \leq t_n < s$. Let us suppose that the first $m$ are before $T_0$ and others after. We have that \begin{align*}
        &\E_\Q(\phi(w_{t_1}, \dots w_{t_n})(M_t-M_{s}))=\\
        &\int_{\Omega_0}\int_\R \int_{\Omega_0}  F(w^1_{t_1}, \dots, 
        w^1_{t_m}, w^2_{t_{m+1}-T_0}, \dots,w^2_{t_{n}-T_0} ) \varphi 
        (w^2) d\Q^2_y(dw^2)p^1_{T_0}(y)dy \Q^1(dw^1).\\
        \end{align*}
        Since $\Q^2$ solves the $(MP(p^1_{T_0},T_0,\tilde{b} + b_1))$, 
        one has that $\E_{\Q^2}(\phi'(w^2_{t_1'}, \dots w^2_{t_n'}) 
        F)=0$ for any continuous bounded functional $\phi'$ on $\R^n$, 
        any $n\in N$ and any ${t_1'}\leq \dots \leq {t_n'}<s-T_0.$ 
        Taking $\phi'(w^2_{t_1'}, \dots w^2_{t_n'})=  \phi(w^1_{t_1}, 
        \dots, w^1_{t_m}, w^2_{t_{m+1}-T_0}, \dots,w^2_{t_{n}-T_0} )$ 
        for a fixed $x^1$, we conclude that
         $$\int_{\R} \int_{\Omega_0} \phi(w^1_{t_1}, \dots, w^1_{t_m}, 
         w^2_{t_{m+1}-T_0}, \dots,w^2_{t_{n}-T_0} ) \varphi 
         (w^2)\Q_y^2(dw^2)p^1_{T_0}(y)dy=0.$$
        Therefore,
   $$\E^\Q(\phi(w_{t_1}, \dots w_{t_n})(M_t-M_{s}))=0.$$
   Thus, $\E^\Q(M_t|\B_{s})=M_{s}$ for $T_0\leq s\leq t\leq 2 T_0$.
  \end{enumerate}
\end{enumerate}
To summarize the preceding, we have just shown the existence of a solution to $(MP(p_0,2T_0,b))$. Finally, we proceed as in the proof of Theorem \ref{localEX} to deduce the existence and uniqueness of a weak solution to \eqref{NLSDEsimple} up to $2T_0$.
% with the desired property on its time marginals.
%Now, we conclude the existence of a weak solution to \eqref{NLSDEsimple} up to $2T_0$. As the one dimensional marginals of this solution are densities that satisfy \eqref{density_space}, in view of Corollary \ref{densityLp}, Inequality \ref{density_infty} is satisfied as soon as one concludes that the drift of $X$ is bounded. \\\\
%As for weak uniqueness, Section \ref{subs:uniq} holds true for any $T>0$, so we can use the same arguments as in the proof of Theorem \ref{localEX}. This concludes the proof of Proposition \ref{localEx2T0}.

\subsection{End of the proof of Theorem \ref{mainTh}: construction of 
the global solution}
Given any finite time horizon $T>0$, split the interval~$[0,T]$ into 
$n=[\frac{T}{T_0}]+1$ intervals of length not exceeding $T_0$ and repeat $n$ times 
the procedure used in the preceding subsection. By construction, the 
time 
marginals of this solution to $(MP(p_0,T,b))$ has probability densities 
which satisfy~\eqref{density_space}.
% The uniqueness part can be repeated since it is independent of $T>0$ as long as the marginals satisfy \eqref{density_space}.
% Notice that this procedure was done for $\sigma =1$. Once we have unique solution in that case, by a simple transformation we also have unique solution for the case $\sigma > 0$. 
\begin{remark}
Using similar arguments as above one can construct a 
solution to \eqref{NLSDEsimple} when the initial condition $p_0$ is in 
$L^\infty(\R) \cap L^1(\R)$ or, respectively, $L^2(\R) \cap L^1(\R)$. 
In these cases we use Remark~\ref{rk:regularity-result} in the 
iterative procedure. Consequently, the weak solution is 
unique under the constraint 
that the one dimensional marginal densities 
$(p_t)_{t\leq T}$ belong to $L^\infty((0,T);L^\infty(\R) \cap L^1(\R))$
or, respectively, satisfy
$$\|p_t\|_{L^\infty(\R)}\leq \frac{C_T}{t^{1/4}}.$$
\end{remark}
\subsection{Proof of Corollary \ref{cor:forPS}}
As \eqref{density_infty} implies \eqref{density_space}, Theorem \eqref{mainTh} implies the existence of a solution to \eqref{NLSDEsimple} in the modified sense.

Let $(\Omega,\F,\P,(\F_t),X,W)$ be any solution to \eqref{NLSDEsimple} in this new sense. Denote by $(p_t)_{t\leq T}$ the time marginal densities of $\P \circ X^{-1}$. Then, \eqref{density_space} and (H-7) imply that $B(t,x;p)$ is a bounded function on $[0,T]\times \R$. In view of Corollary \ref{prel:densityEST}, $(p_t)_{t\leq T}$ satisfies \eqref{density_infty} and as such it is a solution to \eqref{NLSDEsimple} in the sense of Definition \ref{def:weakSDE}. Theorem \ref{mainTh} implies the uniqueness of this solution in the modified sense.
\section{Application to the one-dimensional Keller--Segel model}
\label{sec:appl-K-S}
In this section we prove Corollary~\ref{KSmain}. We 
start with checking that $K^\sharp$ satisfies Hypothesis~(H).
The condition (H-1) is satisfied since for $t>0$ one has
$$ \|K^\sharp_t\|_{L^1(\R)} = \frac{C}{\sqrt{t}}\int 
|z|e^{-\frac{z^2}{2}}dz .$$
%~~~\text{and}~~~\|K^\sharp_t\|_{L^2(\R)} \leq \frac{C_T}{t^{3/4}}.
From the definition of $K^\sharp$ it is clear that for $t>0$, 
$K^\sharp_t$ is a bounded and continuous function on $\R$.
The condition (H-3) is also obviously satisfied.
As already noticed,
$$\|K^\sharp_{t-s}\|_{L^1(\R)}=\frac{C}{\sqrt{t-s}}, $$
from which, 
%by definition of the $\beta$-function,
$$f_1(t) := \int_0^t \frac{\|K^\sharp_{t-s}\|_{L^1(\R)}}{\sqrt{s}}ds = 
C\int_0^t\frac{1}{\sqrt{s}\sqrt{t-s}}ds = C\int_0^1 \frac{1}{\sqrt{x}\sqrt{1-x}} dx=C, $$
where $C$ is a universal constant.
% fin de rouge
Now let $\phi$ be a probability density  on $\R$. For $(t,x)\in 
(0,T]\times \R$, one has 
\begin{equation*}
\begin{split}
\int \phi(y) \|K^\sharp_\cdot(x-y)\|_{L^1(0,t)}~dy &\leq C \int 
\phi(y)|x-y| \int_0^t \frac{1}{s^{3/2}}e^{-\frac{|x-y|^2}{2s}}ds~dy\\
&= \int \phi(y) |x-y|\int_{\frac{|x-y|}{\sqrt{t}}}^\infty 
\frac{z^3}{|x-y|^3}e^{-\frac{z^2}{2}}\frac{|x-y|^2}{z^3}dz~dy \\
&= \int \phi(y) \int_{\frac{|x-y|}{\sqrt{t}}}e^{-\frac{z^2}{2}}dz~dy.
\end{split}
\end{equation*}
This shows that (H-5) is satisfied.
Finally, to prove (H-6) we notice that for every $t\in [0,T]$ 
$$\int_0^T \| K^\sharp_{T+t-s} 
\|_{L^1(\R)}~\frac{1}{\sqrt{s}}~ds \leq \int_0^T \frac{C}{\sqrt{T+t-s} 
\sqrt{s}} ds \leq  C \int_0^T \frac{1}{\sqrt{T-s} 
\sqrt{s}} ds=C.$$
Therefore,
%since $ c_0 \in  C^1_b(\R)$, we have
%$$|\chi e^{-\lambda t} \E c'_0(x+W_t)|\leq \chi  \| 
%c_0'\|_{L^\infty(\R)}.$$
%In addition, it is continuous w.r.t. the space variable. Thus, 
in view 
of Theorem \ref{mainTh}, Equation \eqref{NLSDEd} with $d=1$ is 
well-posed.\footnote{With similar calculations as for $f_1$, one easily checks that the function $f_2$ is 
bounded on any compact time interval. Thus, Corollary \ref{cor:forPS} applies as well as Theorem \ref{mainTh}.} 

Denote by $\rho(t,x)\equiv p_t(x)$ the time marginals of the constructed probability distribution. Now, define the function $c$ as in~(\ref{eq:c-KS}).
%$$ c_t := e^{-\lambda t}g(t,\cdot)\ast c_0 + \int_0^t 
%e^{-\lambda s} \rho_{t-s}\ast g(s,\cdot)~ds. $$
In view of Inequality~\eqref{density_infty}, for any $t\in (0,T]$
the function $c(t,\cdot)$ is well defined and bounded continuous.
%, and $$\|c_t\|_{L^\infty(\R)} \leq \|c_0\|_{L^\infty(\R)} + C_T.$$
Let us show that $c \in L^\infty([0,T];C_b^1(\R))$. 

We have
$$\frac{\partial c}{\partial x}(t,x)= \frac{\partial}{\partial x} 
\left(e^{-\lambda t}\E (c_0(x+W_t)\right)+ \frac{\partial}{\partial x} 
\left( \E \int_0^t e^{-\lambda s} \rho(t-s, x+W_s)ds\right).$$
Then observe that
\begin{equation*}
\begin{split}
\frac{\partial}{\partial x}\E \int_0^t e^{-\lambda s} 
\rho(t-s,x+W_s)ds &= \frac{\partial}{\partial x}\int_0^t e^{-\lambda 
s} 
\int \rho(t-s,x+y)\frac{1}{\sqrt{2\pi s}}e^{\frac{-y^2}{4s}}dy ds\\
&= \frac{\partial}{\partial x}\int_0^t e^{-\lambda (t-s)} \int 
\rho(s,y)\frac{1}{\sqrt{2\pi (t-s)}}e^{\frac{-(y-x)^2}{4(t-s)}} 
dyds \\
&=:\frac{\partial}{\partial x}\int_0^t f(s,x) ds.
\end{split}
\end{equation*}
As for any $0<s<t$
$$|\frac{\partial}{\partial x}\frac{1}{\sqrt{ 
t-s}}e^{\frac{-(y-x)^2}{2(t-s)}} |\leq 
\frac{|y-x|}{2(t-s)^{3/2}}e^{\frac{-(y-x)^2}{2(t-s)}}\leq 
\frac{C}{t-s},$$
we have
$$\frac{\partial f}{\partial x}(s,x)= e^{-\lambda (t-s)} \int 
\rho(s,y)\frac{y-x}{2\sqrt{2\pi}(t-s)^{3/2}}e^{\frac{-(y-x)^2}{2(t-s)}}
 dy.$$
Now, we repeat the same argument for $\frac{\partial}{\partial 
x}\int_0^tf(s,x)ds$. In order to justify the differentiation under the  
integral sign we notice that
$$|\frac{\partial f}{\partial x}(s,x)|\leq \frac{C_T}{ 
\sqrt{(t-s)s}}.$$
Gathering the preceding calculations we have obtained
$$\frac{\partial c}{\partial x}(t,x)=  e^{-\lambda t}\E 
c_0'(x+W_t) + \int_0^t e^{-\lambda (t-s)} \int 
\rho_{s}(y)\frac{y-x}{\sqrt{2\pi}(t-s)^{3/2}}e^{\frac{-(y-x)^2}{2(t-s)}}
 dy~ds.$$
Using the assumption on $c_0$ and Inequality~\eqref{density_infty}, for 
any $t\in (0,T]$ one has
$$\|\frac{\partial c}{\partial x}(t,\cdot)\|_{L^\infty(\R)} \leq 
\|c_0'\|_{L^\infty(\R)}  + C_T.$$
In addition, the preceding calculation and Lebesgue's Dominated 
Convergence Theorem show that
$\frac{\partial c}{\partial x}(t,\cdot)$ is continuous on $\R$. We thus have 
obtained the desired property.

%As $\chi \frac{\partial}{\partial x}c(t,x)$ is the drift coefficient
%in the SDE~\eqref{NLSDEsimple},
The above discussion shows that we are in a position to apply 
Proposition 
~\ref{MeqUniq} 
with $b(t,x) \equiv  \chi e^{-\lambda t}\E 
c_0'(x+W_t) $ and
$B(t,x;\rho)$ defined as in~(\ref{def:B}) with $K\equiv K^\sharp$:
the function $\rho(t,x)$ satisfies~(\ref{eq:rho-KS})
%$$\rho_t= g(t, \cdot \ )\ast \rho_0+\chi\int_0^t\frac{\partial 
%g}{\partial x} 
%(t-s, \cdot \ )\ast( \frac{\partial}{\partial x}c(s,\cdot \ )\rho_s)ds 
%$$
in the sense of the distributions.
% fin de rouge
Therefore, it is a solution to the Keller Segel system 
\eqref{KS} in the sense of Definition~\ref{notionOfSol}.
We now check the uniqueness of this solution.

Assume there exists another solution $\rho^1$ satisfying Definition 
~\ref{notionOfSol} with the same initial condition as $\rho$. 
For notation convenience, in the calculation below we set
$c_t(x):=c(t,x)$, $c^1_t(x):=c^1(t,x)$, $\rho_t(x):=\rho(t,x)$, and 
$\rho^1_t(x):=\rho^1(t,x)$.
%As  $(\rho^1,c^1)$  satisfies Definition \eqref{notionOfSol},
%$$\rho^1 \in L^\infty([0,T]; L^1(\R)), \quad c^1 \in 
%L^\infty([0,T];C_b^1(\R)), \quad \text{and} \quad \forall 0<t\leq T, 
%~~  \|\rho^1_t\|_{L^{2}(\R)}\leq \frac{C}{t^{\frac{1}{4}}},$$
%and $\forall t\in(0,T]$
%\begin{equation*}
%    \begin{cases}
%    & \rho_t^1 =g(t,\cdot)\ast \rho_0+\int_0^t\frac{\partial}{\partial 
%    y} g_{t-s}\ast( \frac{\partial}{\partial x} c^1(s,\cdot \ 
%    )\rho_s^1)ds\\
%    & c_t^1 = e^{-\lambda t}g(t,\cdot)\ast c_0+ \int_0^t e^{-\lambda 
%s} 
%    \rho^1_{t-s}\ast g(s,\cdot) ds.\\
%    \end{cases}
%     \end{equation*}
%We will prove by Gronwall argument that for any $t\in (0,T]$, one has 
%$\|\rho^1_t-\rho_t\|_1=0$. 

%As $g(t, \cdot \ )\leq 
%\frac{C_T}{\sqrt{t}}$ and $\|\frac{\partial}{\partial y} g_{t-s} 
%\|_{L^2(\R)}\leq \frac{C_T}{(t-s)^{3/4}}$ one has
%$$\|\rho_t^1\|_{L^\infty(\R)}\leq \frac{C}{\sqrt{t}}+ \sup_{0\leq 
%s\leq 
%T} \|\frac{\partial}{\partial x}c^1_t\|_{L^\infty(\R)}\int_0^t 
%\frac{C}{s^{1/4} (t-s)^{3/4}}ds\leq \frac{C_T}{\sqrt{t}}.$$

%from here everything is the same starting from the following line: 
Using Definition~\ref{notionOfSol},
\begin{align*}
&\|\rho^1_t-\rho_t\|_{L^1(\R)}\leq \int_0^t \|\frac{\partial 
g_{t-s}}{\partial x} \ast( \frac{\partial c^1_s}{\partial 
x}\rho_s^1-\frac{\partial c_s}{\partial x}\rho_s) 
\|_{L^1(\R)}ds\\
&\leq \int_0^t \|\frac{\partial g_{t-s}}{\partial x}\ast( 
\frac{\partial c^1}{\partial x}(\rho_s^1-\rho_s))\|_{L^1(\R)}ds 
+\int_0^t \|\frac{\partial g_{t-s}}{\partial x} 
\ast(\rho_s(\frac{\partial c^1}{\partial x} 
-\frac{\partial c_s}{\partial x})) \|_{L^1(\R)}ds\\
& =:I+II.
\end{align*}
Using standard convolution inequalities and $\|\frac{\partial 
g_{t-s}}{\partial x}\|_{L^1(\R)}\leq \frac{C}{\sqrt{t-s}}$ we deduce:
$$I\leq C \int_0^t \frac{\|\rho_s^1-\rho_s\|_{L^1(\R)}}{\sqrt{t-s}} 
ds\quad\text{      and      } \quad II\leq C \int_0^t 
\frac{\|\frac{\partial c^1_s}{\partial x} -\frac{\partial c_s}{\partial 
x} \|_{L^1(\R)}}{\sqrt{t-s}\sqrt{s}} ds.$$
Therefore
\begin{equation} \label{ineq:c-rho}
\|\frac{\partial c^1_s}{\partial x} 
 -\frac{\partial c_s}{\partial x} \|_{L^1(\R)} \leq \int_0^s \| 
(\rho_u^1-\rho_u)\ast\frac{\partial g_{s-u}}{\partial 
x}\|_{L^1(\R)}du\leq 
C\int_0^s \frac{\|\rho_u^1-\rho_u\|_{L^1(\R)}}{\sqrt{s-u}} du,
\end{equation}
from which
\begin{equation*}
\begin{split}
II &\leq C 
\int_0^t\frac{1}{\sqrt{s}\sqrt{t-s}}
\int_0^s\frac{\|\rho_u^1-\rho_u\|_{L^1(\R)}}{\sqrt{s-u}}~du~ds \\
 &\leq C \int_0^t\|\rho_u^1-\rho_u\|_{L^1(\R)}\int_u^t 
 \frac{1}{\sqrt{s}\sqrt{s-u}\sqrt{t-s}}~ds~du
 \leq C_T \int_0^t \frac{\|\rho_u^1-\rho_u\|_{L^1(\R)}}{\sqrt{u}}du.
\end{split}
\end{equation*}
%since $\int_u^t \frac{1}{\sqrt{s-u}\sqrt{t-s}}ds=C$. 
Gathering the preceding bounds for~$I$ and $II$ we get
$$\|p^1_t-p_t\|_{L^1(\R)}\leq C_T \int_0^t \frac{\|p_s^1-p_s\|_{L^1(\R)}}{\sqrt{t-s}}ds 
+ C_T \int_0^t \frac{\|p_s^1-p_s\|_{L^1(\R)}}{\sqrt{s}}ds.$$
Lemma \ref{ModifiedSgGronwall} implies that
$\|\rho^1_t-\rho_t\|_{L^1(\R)}=0$ for every $t\leq T$. In view 
of~\eqref{ineq:c-rho}
we also have $\|\frac{\partial c^1_t}{\partial x} 
 -\frac{\partial c_t}{\partial x} \|_{L^1(\R)}=0$.
%Thus, for all $t\leq T$, $p_t$ and $p^1_t$ are equal a.e. 
%In that case, it is obvious that 
This completes the proof of 
Corollary~\ref{KSmain}.

\section{Appendix} \label{sec:appendix}
We here propose a light simplification of the calculations in~\cite{QianZheng}.

\begin{proposition} \label{QZTH2}
Let $y\in \R$ and let $\beta$ be a constant.  Denote by 
$p^\beta_y(t,x,z)$ the transition probability density (with respect to 
the Lebesgue measure) of the unique weak solution to 
$$ X_t= x+\beta\int_0^t \text{sgn}(y-X_s)~ds + W_t. $$
Then 
\begin{equation} \label{eq:p-beta-x-y-z}
\begin{split}
p^\beta_y(t,x,z) &= \frac{1}{\sqrt{2\pi} t^{3/2}}\int_0^\infty 
e^{\beta(|y-x| + \bar{y} - |z-y|)- 
\frac{\beta^2}{2}t}(\bar{y}+|z-y|+|y-x|)
e^{-\frac{(\bar{y}+|z-y|+|y-x|)^2}{2t}} d\bar{y} \\
&\qquad +\frac{1}{\sqrt{2\pi t}} e^{\beta(|y-x| - |z-y|)- 
\frac{\beta^2}{2}t}(e^{-\frac{(z-x)^2}{2t}}
- e^{-\frac{(|z-y|+|y-x|)^2}{2t}}).
\end{split}
\end{equation}
In particular,
\begin{equation} \label{eq:p-beta}
p^\beta_y(t,x,y) =  \frac{1}{\sqrt{2\pi t}} 
\int_{\frac{|x-y|}{\sqrt{t}}}^{\infty}  
z e^{-\frac{(z-\beta \sqrt{t})^2}{2}}~dz.
\end{equation}
\end{proposition}

\begin{proof}
Let $f$ be a bounded continuous function.
The Girsanov transform leads to
$$ \E (f(X_t))= \E (f(x+W_t)e^{\beta \int_0^t \text{sgn} (y-x-W_s)dWs- 
\frac{\beta^2}{2}t}). $$
Let $L_t^a$ be the Brownian local time. By Tanaka's formula 
(\cite{KaratzasShreve}, p. 205):
$$|W_t-a|= |a|+\int_0^t \text{sgn}(W_s-a)dW_s + L_t^a.$$
Therefore
$$ \int_0^t \text{sgn}(y-x-W_s)dW_s=  |y-x| + L_t^a - |W_t-(y-x)|, $$
from which
$$\E (f(X_t)) = \E (f(x+W_t)e^{\beta( |y-x| + L_t^{y-x} - |W_t-(y-x)|)
- \frac{\beta^2}{2}t}). $$
Recall that $(W_t,L_t^a)$ has the following joint distribution (see 
\cite[p.200,Eq.(1.3.8)]{BorodinSalminen}:
\begin{equation*}
\begin{cases}
& \bar{y}>0 : \quad \P(W_t \in dz, L_t^a\in 
d\bar{y})=\frac{1}{\sqrt{2\pi} t^{3/2}} 
(\bar{y}+|z-a|+|a|)e^{-\frac{(\bar{y}+|z-a|+|a|)^2}{2t}} d\bar{y}dz.\\
& \P (W_t \in dz, L_t^a=0)= \frac{1}{\sqrt{2\pi t}} 
e^{-\frac{z^2}{2t}}dz -\frac{1}{\sqrt{2\pi t}}
e^{-\frac{(|z-a|+|a|)^2}{2t}}dz .
\end{cases}
\end{equation*}
It comes:
\begin{equation*}
\begin{split}
\E (f(X_t)) &= \frac{1}{\sqrt{2\pi} t^{3/2}} \int_\R \int_0^\infty 
f(x+z) e^{\beta(|y-x| +  \bar{y} - |z-(y-x)|)
- \frac{\beta^2}{2}t}(\bar{y}+|z-(y-x)|+|y-x|) \\
&\qquad ~~~~~~~~~~~~~~~~~~~~~
e^{-\frac{(\bar{y}+|z-(y-x)|+|y-x|)^2}{2t}}~d\bar{y}~dz \\
&\qquad + \frac{1}{\sqrt{2\pi t}} \int_\R f(x+z) e^{\beta(|y-x|  
- |z-(y-x)|)- \frac{\beta^2}{2}t}(e^{-\frac{z^2}{2t}}
-e^{-\frac{(|z-(y-x)|+|y-x|)^2}{2t}})~dz.
\end{split}
\end{equation*}
The change of variables $x+z=z'$ leads to
\begin{equation*}
\begin{split}
\E (f(X_t)) &= \frac{1}{\sqrt{2\pi} t^{3/2}} \int_\R f(z') 
\int_0^\infty e^{\beta(|y-x|  + \bar{y} - |z'-y|)- 
\frac{\beta^2}{2}t}(\bar{y}+|z'-y)|+|y-x|)
e^{-\frac{(\bar{y}+|z'-y|+|y-x|)^2}{2t}}~d\bar{y}dz' \\
&\qquad + \frac{1}{\sqrt{2\pi t}} \int_\R f(z') e^{\beta(|y-x| - |z'-y|)
- \frac{\beta^2}{2}t}(e^{-\frac{(z'-x)^2}{2t}}
- e^{-\frac{(|z'-y|+|y-x|)^2}{2t}})dz',
\end{split}
\end{equation*}
from which the desired result follows.

%Rewrite it as
%\begin{align*}
%&p^\beta_y(t,x,z)=\frac{1}{\sqrt{2\pi} t^{3/2}}
% e^{-2\beta |z-y|}\int_0^\infty 
% (\bar{y}+|z-y|+|y-x|)e^{-\frac{(\bar{y}+|z-y|+|y-x|-\beta t)^2}{2t}} 
%%%d\bar{y} \\
%& + \frac{1}{\sqrt{2\pi t}} e^{\beta(|y-x| - |z'-y|)- 
%%%\frac{\beta^2}{2}t}(e^{-\frac{(z-x)^2}{2t}}-e^{-\frac{(|z-y|+|y-x|)^2}{2t}}).
%\end{align*}
%Evaluating the previous expression in $z=y$,
%$$p^\beta_y(t,x,y)= \frac{1}{\sqrt{2\pi} t^{3/2}}\int_0^\infty 
% (\bar{y}+|y-x|)e^{-\frac{(\bar{y}+|y-x|-\beta t)^2}{2t}} d\bar{y}.$$
%Applying the change of variables $\frac{\bar{y}+|y-x|}{\sqrt{t}}=z$
%$$p^\beta_y(t,x,y)= \frac{1}{\sqrt{2\pi t}}
%\int_{\frac{|y-x|}{\sqrt{t}}}^\infty z 
%e^{-\frac{(z-\beta \sqrt{t})^2}{2}} dz.$$
\end{proof}

In the next Corollary we use the same notation as in the proof of Theorem \ref{QianZhengUSTh}.
\begin{cor}
\label{cor:ineq-derivee-densite}
Let $0<s<t\leq T$.Then for any $z,y \in \R$, there exists $C_{T,\beta,x,y}$ such that
$$\E |\left(\frac{\partial}{\partial x}p^\beta_y\right)(t-s,X_s^{(b)},z)|\leq C_{T,\beta,x,y} h(s,z),$$
where $h$ belongs to $L^1([0,t]\times \R)$.
%For any $w_0\in\R$ and $x$ in a compact interval 
%$w_0-\delta,w_0+\delta$
%with $\delta>0$ one has
%$$ \left| \frac{\partial p^\beta_y(t,x,z)}{\partial x}(t,x,z) \right|
%\leq \exp(-\beta~|z-y|+\beta~|y-w_0|+\delta). $$
\end{cor}

%\begin{proof}
%Observe that
%\begin{equation*}
%\begin{split}
%\frac{\partial}{\partial 
%x}p^\beta_y(t-s,\bar{x},z) &= \frac{\beta}{\sqrt{2\pi(t-s)} }e^{-2\beta 
%|z-y|}e^{-\frac{(|z-y|+|y-\bar{x}|-\beta 
%(t-s))^2}{2(t-s)}} \text{sgn}(\bar{x}-y) \\
%&\qquad  + \frac{\beta}{\sqrt{2\pi (t-s)}} e^{-\beta  |z-y|- 
%\frac{\beta^2}{2}(t-s)}e^{\beta|y-\bar{x}|
%-\frac{(z-\bar{x})^2}{2(t-s)}}
%\text{sgn}(\bar{x}-y) \\
%&\qquad  + \frac{z-\bar{x}}{\sqrt{2\pi (t-s)^{3/2}}} e^{-\beta  |z-y|- 
%\frac{\beta^2}{2}(t-s)}e^{\beta|y-\bar{x}|
%-\frac{(z-\bar{x})^2}{2(t-s)}}.
%\end{split}
%\end{equation*}
%The sum of the absolute values of the first two terms in the right-hand 
%side is bounded from above by
%$$ \frac{\beta}{\sqrt{2\pi(t-s)}}e^{-2\beta |z-y| 
%+ \beta |y-\bar{x}|}. $$
%The absolute value of the last term is bounded by 
%$$ \frac{1}{\sqrt{2\pi\sqrt{t-s}}} e^{-\beta|z-y| + \beta|y-\bar{x}|} 
%$$
%since for some $\alpha>0$ one has $|a|^4e^{-a^2}\leq Ce^{-\alpha a^2}$ 
%for any $a\in\R$. The desired result follows.
%\end{proof}
\begin{proof}

%As the drift is bounded, by Girsanov transformation, we have
%\begin{equation}
%\label{lemmaDerOfDensBound:girsanov}
%\E (|\frac{\partial}{\partial x}p^\beta_y(t-s,X_s^{(b)},z)|)\leq C_\beta \sqrt{\E(|\frac{\partial}{\partial x}p^\beta_y(t-s,W_s^x,z)|^2)}.
%\end{equation}
%As we have seen, the derivative in the second variable of $p^\beta_y(t-s,\bar{x},z)$ has two components:
%\begin{align*}
%&\frac{\partial}{\partial x}p^\beta_y(t-s,\bar{x},z)=\frac{\beta}{\sqrt{2\pi (t-s)} }e^{-2\beta |z-y|}e^{-\frac{(|z-y|+|y-\bar{x}|-\beta (t-s))^2}{2(t-s)}}sgn(\bar{x}-y)+ \\
%&  \frac{1}{\sqrt{2\pi (t-s)}} e^{-\beta  |z-y|- \frac{\beta^2}{2}(t-s)}e^{\beta|y-\bar{x}|-\frac{(z-\bar{x})^2}{2(t-s)}}(\beta sgn(\bar{x}-y)+\frac{z-\bar{x}}{t-s}). \\
%\end{align*}
%We can directly bound the first one such that the lemma is satisfied,
%$$|\frac{\beta}{\sqrt{2\pi (t-s)} }e^{-2\beta |z-y|}e^{-\frac{(|z-y|+|y-\bar{x}|-\beta (t-s))^2}{2(t-s)}}sgn(\bar{x}-y)|\leq \frac{\beta}{\sqrt{2\pi (t-s)} }e^{-2\beta |z-y|}.$$
By Girsanov's theorem,
for some constant $C_{T,\beta}$ we have
\begin{equation*} 
%\label{lemmaDerOfDensBound:girsanov}
\E \left|\left(\frac{\partial}{\partial x}p^\beta_y\right)(t-s,X_s^{(b)},z) \right|
\leq C_{T,\beta} \sqrt{\E \left|\left(\frac{\partial}
{\partial x}p^\beta_y\right)(t-s,W_s^x,z) \right|^2}.
\end{equation*}
Observe that
\begin{equation*}
\begin{split}
\frac{\partial}{\partial 
\bar{x}}p^\beta_y(t-s,\bar{x},z) &= \frac{\beta}{\sqrt{2\pi(t-s)} }e^{-2\beta 
|z-y|}e^{-\frac{(|z-y|+|y-\bar{x}|-\beta 
(t-s))^2}{2(t-s)}} \text{sgn}(\bar{x}-y) \\
&\qquad  + \frac{\beta}{\sqrt{2\pi (t-s)}} e^{-\beta  |z-y|- 
\frac{\beta^2}{2}(t-s)}e^{\beta|y-\bar{x}|
-\frac{(z-\bar{x})^2}{2(t-s)}}
\text{sgn}(\bar{x}-y) \\
&\qquad  + \frac{z-\bar{x}}{2\pi (t-s)^{3/2}} e^{-\beta  |z-y|- 
\frac{\beta^2}{2}(t-s)}e^{\beta|y-\bar{x}|
-\frac{(z-\bar{x})^2}{2(t-s)}}.
\end{split}
\end{equation*}
The sum of the absolute values of the first two terms in the right-hand 
side is bounded from above by
$$ \frac{\beta}{\sqrt{2\pi(t-s)}}e^{-2\beta |z-y| 
+ \beta |y-\bar{x}|}. $$
Thus, 
$$\E \left|\left(\frac{\partial}{\partial x}p^\beta_y\right)(t-s,X_s^{(b)},z) \right|
\leq \frac{C_{T,\beta}}{\sqrt{2\pi(t-s)}}\sqrt{\E e^{2\beta |y-W_s^x|}}+ \frac{C_{T,\beta}}{(t-s)^{3/2}}\sqrt{\E(|z-W_s^x|^2e^{2\beta |y-W_s^x|-\frac{(z-W_s^x)^2}{t-s}})}=: B+A.$$
%Let us treat the following term:
%$$A:=\frac{C_\beta}{(t-s)^{3/2}}\sqrt{\E(|z-W_s^x|^2e^{2\beta |y-W_s^x|-\frac{(z-W_s^x)^2}{t-s}})}.$$
Notice that
$$A\leq \frac{C_{T,\beta}}{(t-s)^{3/2}}(\E[|z-W_s^x|^4e^{-2\frac{(z-W_s^x)^2}{t-s}}]\E[e^{4\beta|y-W_s^x|}])^{1/4}=:\frac{C_{T,\beta}}{(t-s)^{3/2}}(A_1A_2)^{1/4}.$$ 
Firstly, as there exists an $\alpha>0$ such that $|a|^4e^{-a^2}\leq Ce^{-\alpha a^2}$, one has
$$A_1\leq C (t-s)^2 \int e^{-\alpha \frac{(z-u)^2}{t-s}}g_s(u-x)du \leq \frac{(t-s)^{2+1/2}}{\sqrt{s+(t-s)/(2\alpha)}} e^{-\frac{(z-x)^2}{2(s+(t-s)/(2\alpha))}}.$$ 
Secondly,
\begin{align*}
A_2&=\int e^{4\beta|y-u|}g_s(u-x)du= e^{-4\beta y}\int_y^\infty e^{4\beta u}\frac{1}{\sqrt{s}}e^{-\frac{(u-x)^2}{2s}}du+e^{4\beta y}\int_{-\infty}^y e^{-4\beta u}\frac{1}{\sqrt{s}}e^{-\frac{(u-x)^2}{2s}}du\\
&= e^{4\beta(x- y)}e^{8\beta^2 s }\int_y^\infty \frac{1}{\sqrt{s}}e^{-\frac{(u-x-4\beta s)^2}{2s}}du+e^{4\beta(y-x)}e^{8\beta^2 s }\int_{-\infty}^y \frac{1}{\sqrt{s}}e^{-\frac{(u-x+4\beta s)^2}{2s}}du \leq e^{8\beta^2 s } C_{\beta,x,y}.\\
\end{align*}
%$$A_2=\int e^{4\beta|y-u|}g_s(u-x)du= e^{-4\beta y}\int_y^\infty e^{4\beta u}\frac{1}{\sqrt{s}}e^{-\frac{(u-x)^2}{2s}}du+e^{4\beta y}\int_{-\infty}^y e^{-4\beta u}\frac{1}{\sqrt{s}}e^{-\frac{(u-x)^2}{2s}}du. $$
%Completing the squares in the exponents,
%$$A_2=e^{4\beta(x- y)}e^{8\beta^2 s }\int_y^\infty \frac{1}{\sqrt{s}}e^{-\frac{(u-x-4\beta s)^2}{2s}}du+e^{4\beta(y-x)}e^{8\beta^2 s }\int_{-\infty}^y \frac{1}{\sqrt{s}}e^{-\frac{(u-x+4\beta s)^2}{2s}}du,$$
%we conclude that
%$$A_2\leq e^{8\beta^2 s } C_{\beta,x,y}.$$
Therefore, 
$$A\leq C_{T,\beta,x,y} \frac{1}{(t-s)^{7/8}}g_{s+(t-s)/(2\alpha)}(z-x).$$
The term $B$ is treated in the similar way as $A_2$. 
%Thus, we can bound the expression in the lemma with a function integrable in time and space.		
\end{proof}

%%%%%%%%%%%%%%%%%%%%%%%%%%%%%%%%%%%%%%%%%%%%%%%%%%%%%%%%

%%%%%%%%%%%%%%%%%%%%%%%%%%%%%%%%%%%%%%%%%%%%%%%%%%%%%%%%
\pagebreak
\bibliography{biblio}
\end{document}